\documentclass{amsart}

\makeatletter
 \def\LaTeX{\leavevmode L\raise.42ex
   \hbox{\kern-.3em\size{\sfksize}{0pt}\selectfont A}\kern-.15em\TeX}
\makeatother

\newcommand{\BibTeX}{{\rm B\kern-.05em{\sc
i\kern-.025emb}\kern-.08em\TeX}}
\newtheorem{thm}{Theorem}[section]
\newtheorem{col}[thm]{Corollary}
\newtheorem{lem}[thm]{Lemma}
\newtheorem{rem}[thm]{Remark}
\theoremstyle{definition}
\newtheorem{defn}[thm]{Definition}

\newtheorem{theorem}[thm]{Theorem}

\newcommand\RR{{\Bbb R}}

\newcommand\NN{{\Bbb N}}

\numberwithin{equation}{section}

\begin{document}

\title[Band-limited Localized Parseval frames and Besov spaces
on Compact Homogeneous Manifolds]{Band-limited Localized Parseval frames and Besov spaces
on Compact Homogeneous Manifolds}

\maketitle
\begin{center}

\author {Daryl Geller}
\footnote{Department of Mathematics, Stony Brook University, Stony Brook, NY 11794-3651; darylkmath.sunysb.edu}

\author{Isaac Z. Pesenson }\footnote{ Department of Mathematics, Temple University,
 Philadelphia,
PA 19122; pesenson@temple.edu. The author was supported in
part by the National Geospatial-Intelligence Agency University
Research Initiative (NURI), grant HM1582-08-1-0019. }

\end{center}

\begin{abstract}
 In the last decade,
methods based on various kinds of spherical wavelet bases  have found applications in virtually all areas 
where analysis  of spherical data is required,
including cosmology, weather prediction, and geodesy. 
In particular, the so-called needlets (=band-limited Parseval frames) have become
an  important tool for the analysis of  Cosmic Microwave Background (CMB) temperature data. 
The goal of the present paper is to construct band-limited and highly localized Parseval frames on 
general compact homogeneous 
manifolds. Our construction can be considered as an analogue of the well-known $\varphi$-transform on Euclidean spaces.
\end{abstract}


\section{Introduction}
 In the last decade,
methods based on spherical wavelets have found applications in virtually all areas where analysis  of spherical data is 
required, including cosmology, weather prediction and geodesy (see \cite{FMN}, \cite{F},   \cite{J}, \cite{V1}  and the 
references therein).
In particular, they have become an important tool for the analysis of  Cosmic Microwave Background 
(CMB) temperature data (\cite{G1}, \cite{G2}, \cite{V1}, \cite{V2}, \cite{antoine}, \cite{jfaa1}, \cite{jfaa3},
\cite{Cruz1}, \cite{chmpv}, \cite{mcewen1}, \cite{mcewen2}, \cite{wiaux2}, and many other articles).  
In analyzing CMB temperature data, one seeks precise estimates of several parameters of the greatest
interest for Cosmology and Theoretical Physics, as well as information on possible regions of non-Gaussianity, and other
information as well.

In the past few years, a new kind of wavelet has found many fruitful applications in the analysis of CMB  temperature data, 
the so-called {\em spherical needlets}, which form a Parseval (= normalized tight) frame on the sphere 
(see \cite{W1}, \cite{W2}, \cite{W3} for information about Parseval frames on Euclidean spaces).
Spherical needlets  were introduced in  \cite{narc1}, \cite{narc2}, and then 
used for rigorous statistical analysis of 
spherical random  fields in \cite{BKMP1}, \cite{BKMP2}, \cite{mpbb08} and other articles.  This analysis was 
particularly effective in extracting the desired consequences from CMB temperature data.   

The interest in needlets on spheres can be explained by their nearly optimal space-frequency 
localization properties.  These properties of needlets (and other localized bases, such as the
``Mexican needlets'' of \cite{gmfr}) allow one
to perform frequency analysis of signals (functions), even when one only has partial information about them.
 
For example, the CMB models are best analyzed in the frequency domain, where the behavior 
at different multipoles can be investigated separately; on the other hand, partial sky coverage and other missing observations 
make the evaluation of spherical harmonic transforms impossible.

A recent advance in this area was the development of {\em spin needlets} on the sphere
\cite{spinph},\cite{spinmat},\cite{spinstat},\cite{spinnr}, for the purpose of statistical analysis of 
CMB {\em polarization}, which is also expected to have very significant consequences in physics.  

In a different direction, nearly tight frames, which were smooth and highly localized in both space and frequency,
were developed on general smooth compact manifolds
without boundary, in \cite{gmcw}-\cite{gmbes}.  These frames, as in the case of spherical needlets, were
constructed from the kernels of certain functions of the Laplace-Beltrami operator.  (An analogous 
construction had been carried out earlier for stratified Lie groups with lattice subgroups, in 
\cite{gm1}.) 

In this article, we will show that on compact {\em homogeneous} manifolds, one can do better -- 
one can arrange for the frames arising from these methods to actually be {\em Parseval} (and, of course,
highly localized in space and frequency).  We will also show that one can characterize Besov spaces 
through a knowledge of the size of frame coefficients, thereby generalizing results of \cite{narc2}
for the sphere.  (Our results on frame characterizations of Besov spaces are closely related to those in \cite{gmbes}.)

Our frames are a natural generalization of spherical needlets.
They can be also be regarded as analogous to the well-known $\varphi$-transform \cite{FJ1}.

In addition to the fact that one can find Parseval frames, we offer the following motivations for specializing
to the case of homogeneous manifolds.  First, on such manifolds, there is the possibility of 
finding exact formulas for the frame elements.  Secondly, in theoretical physics -- where
many manifolds are considered -- symmetry is of capital importance.  Third, on homogeneous manifolds,
one has the advantage that all of the frame elements at a particular scale can be obtained from each
other through the group action, in the same manner as standard wavelets at a particular scale on the real line
can be obtained from each other by translation. 

Our frames will be band-limited, and hence smooth. 
 One should understand  that the notion of band-limitedness on a compact manifold is not canonical.
 Consider a (connected) compact smooth Riemannian manifold $\bf{M}$, and a smooth elliptic 
self-adjoint positive  differential operator  $A$ on it. It is known that the spectrum of $A$, as an 
operator in the corresponding space $L_{2}(\bf{M})$, is discrete, nonnegative, 
 and accumulates at infinity.  Call the eigenvalues $\lambda_{0} \leq \lambda_{1}\leq.... $,
where we repeat eigenvalues according to their multiplicities. The space $L_{2}(\bf{M})$ has an  orthonormal basis 
$\ u_{\lambda_{0}}, u_{\lambda_{1}},...$ consisting of eigenfunctions of $A$. 

For a fixed operator $A$ we understand the space of $\omega$-band-limited functions ${\mathbf E}_{\omega}(A)$ 
to be the span of all eigenfunctions $u_{\lambda_{j}}$ such that $\lambda_{j}\leq \omega$.

Formally, then, there is great freedom in the notion of band-limitedness.  
However, for the purposes of this article, all these spaces of band-limited functions are  essentially  equivalent,
in the sense that they give rise to the same Besov spaces $B^{\alpha q}_p(\bf{M})$, at least if $\alpha > 0$ and 
$1 \leq p \leq \infty$
(see Theorem  \ref{ernrmeq} below).

The plan of the paper is as follows. 
In section 2, we review some basic facts about compact homogeneous 
manifolds. In section 3, we discuss properties of band-limited functions associated with
a second-order smooth positive elliptic differential operator ${\mathcal L}$.  It is shown in particular that  if $\bf{M}$ is equivariantly embedded into Euclidean space then the span of eigenfunctions of the operator $\mathcal{L}$  is exactly the set of restrictions to $\bf{M}$ of all polynomials in the ambient  space.  In this section we also give several
equivalent definitions of Besov spaces on the manifold.  In one of the definitions, we use 
a global modulus of continuity, constructed through use of certain vector fields on $\bf{M}$. This definition 
is similar to the original definition of Besov spaces on Euclidean spaces and uses just the notion of smoothness. 
Later in the article, in Theorems \ref{ernrmeq} and \ref{beshom}, we
describe the same spaces in terms of approximations by band-limited functions.  
Thus, as one of the consequences of our results, we obtain a new development of one of the oldest topics of classical harmonic 
analysis: the relationships between smoothness and rate of approximations by band-limited functions. 
Specifically, we show that there exists a complete balance between smoothness expressed in terms of modulus of continuity, 
and the rate of approximation by band-limited functions in all spaces $L_{p}(\bf{M}),$ $ \ 1\leq p\leq \infty$, 
as well as other equivalent definitions of Besov spaces on such manifolds. 

In section 4, we describe Plancherel-Polya inequalities (Corollary \ref{completePlPo200})
and in section 5 we obtain cubature formulas with desirable properties  (Theorem \ref{cubformula}).
In these two sections, we do not use any special properties of homogeneous manifolds or the 
second-order positive elliptic differential operator ${\mathcal L}$; the results hold for any smooth compact manifold,
and for any ${\mathcal L}$.

In section 6, we use the homogeneous manifold structure in an essential way to prove a crucial fact, namely that,
for a particular ${\mathcal L}$, 
the product of two band-limited functions of the same bandwidth $\omega$ is also a band-limited function, with a
certain bandwidth $C\omega$, where $C$ is independent of $\omega$.   On the sphere, this property is familiar 
for spherical harmonics; then one may take ${\mathcal L}$ to be the spherical Laplacian, and one may take $C=2$. 
This property of spherical harmonics was used crucially
in the construction of spherical needlets in \cite{narc1}.  
 The generalization to homogeneous manifolds is similarly
needed in our construction of band-limited Parseval frames.   For more general manifolds, it is not clear how to verify this 
property, or even if it is true.   
 It was conjectured in \cite{M} that this  "product" property holds for Laplace-Beltrami operators on analytic compact manifolds.
If ${\bf M} = G/K$ is a homogeneous manifold, we specifically take
${\mathcal L}$ to be the image of the Casimir operator under the differential of the 
quasiregular representation of $G$ in $L_{2}(\bf{M})$ (see section 2).  The operator $-{\mathcal L}$
is a sum of squares of certain vector fields on $\bf{M}$. In some common  
cases, such as compact symmetric spaces of rank one and all compact Lie groups, this operator $\mathcal{L}$ coincides 
with the corresponding Laplace-Beltrami operator.

A number of the results stated, and methods used, in sections 3-6, are from the articles 
\cite{Pes79}-\cite{Pes09}.

In section 7, we review some of the results of \cite{gmcw} - \cite{gmbes}, where the Laplace-Beltrami operator was
used to construct nearly tight frames, which were then used to characterize Besov spaces.  (The Besov space results
in \cite{gmbes} used, in addition to results from \cite{gmcw} and \cite{gmfr}, methods of Frazier-Jawerth \cite{FJ1} 
and Seeger-Sogge \cite{SS}.)  We argue that the results of \cite{gmcw} - \cite{gmbes}
continue to hold if one uses a general ${\mathcal L}$ in place of the Laplace-Beltrami operator.   The arguments of
section 7 do not use any special properties of homogeneous manifolds.  However, the point is that, if we are
on a homogeneous manifold, we are free to use the ${\mathcal L}$ of section 2 in place of the Laplace-Beltrami operator.

Finally, in section 8, by using the results of sections 5 and 6, we construct our Parseval frames on homogeneous manifolds.
By using the results of section 7, we show that they are highly localized, and that one can use them to characterize
Besov spaces, for the full range of the indices.  It is only in the construction of our Parseval frames that we 
use the results of section 6. 

Let us remark that in \cite{narc2}, approximations by polynomials were considered on the sphere,
while in this article, we consider approximations by band-limited functions.
Although it is known \cite{Pes08} that the span of the eigenfunctions of our operator ${\mathcal L}$ is the same as the span 
of all polynomials when one equivariantly embeds the manifold,
the relation between eigenvalues and degrees of polynomials is unknown (at least in 
the general case).  However, it is easy to verify that for compact two-point homogeneous manifolds,  the span of those
eigenfunctions whose eigenvalues are not greater than a value $\ell^{2}, \ \ell\in \mathbb{N},$ is the same as the span of all 
polynomials of degree at most $\ell$. Thus, on compact two-point homogeneous manifolds, our results about approximations by 
band-limited functions can be reformulated in terms of approximations by polynomials.

\section{Compact homogeneous manifolds}

We review some very basic notions of harmonic analysis on
compact homogeneous manifolds \cite{H3}, Ch. II. More details on
this subject can be found, for example, in
 \cite{V}, \cite{Z}.

 Let ${\bf M},\ dim {\bf M}=n,$ be a
compact connected $C^{\infty}$-manifold. One says  that a compact
Lie group $G$ effectively acts on ${\bf M}$ as a group of
diffeomorphisms if:

1)  every element $g\in G$ can be identified with a diffeomorphism
$$
g: {\bf M}\rightarrow {\bf M}
$$
of ${\bf M}$ onto itself and
$$
g_{1}g_{2}\cdot x=g_{1}\cdot(g_{2}\cdot x),\ g_{1}, g_{2}\in G,\
x\in {\bf M},
$$
where $g_{1}g_{2}$ is the product in $G$ and $g\cdot x$ is the
image of $x$ under $g$,

2) the identity $e\in G$ corresponds to the trivial diffeomorphism
\begin{equation}
e\cdot x=x,
\end{equation}

3) for every $g\in G,\ g\neq e,$ there exists a point $x\in {\bf M}$ such
that $g\cdot x\neq x$.

\bigskip

A group $G$ acts on ${\bf M}$ \textit{transitively} if in addition to
1)- 3) the following property holds:

4) for any two points $x,y\in {\bf M}$ there exists a diffeomorphism
$g\in G$ such that
$$
g\cdot x=y.
$$

A \textit{homogeneous} compact manifold ${\bf M}$ is a
$C^{\infty}$-compact manifold on which a compact
Lie group $G$ acts transitively. In this case ${\bf M}$ is necessary of the form $G/K$,
where $K$ is a closed subgroup of $G$. The notation $L_{p}({\bf M}),
1\leq p\leq \infty,$ is used for the usual Banach  spaces
$L_{p}({\bf M},dx), 1\leq p\leq \infty$, where $dx$ is an invariant
measure.

Every element $X$ of the (real) Lie algebra of $G$ generates a vector
field on ${\bf M}$, which we will denote by the same letter $X$. Namely,
for a smooth function $f$ on ${\bf M}$ one has
$$
 Xf(x)=\lim_{t\rightarrow 0}\frac{f(\exp tX \cdot x)-f(x)}{t}
 $$
for every $x\in {\bf M}$. In the future we will consider on ${\bf M}$ only
such vector fields. The translations along integral curves of such
vector fields $X$ on ${\bf M}$  can be identified with a one-parameter
group of diffeomorphisms of ${\bf M}$, which is usually denoted as $\exp
tX, -\infty<t<\infty$. At the same time, the one-parameter group
$\exp tX, -\infty<t<\infty,$ can be treated as a strongly
continuous one-parameter group of operators acting on the space $L_{p}({\bf M}),
1\leq p\leq \infty$.  These operators act on functions according to the
formula
$$
f\rightarrow f(\exp tX\cdot x), \        t\in \mathbb{R}, \
  f\in L_{p}({\bf M}),\        x\in {\bf M}.
$$
 The
generator of this one-parameter group will be denoted by $D_{X,p}$,
and the group itself will be denoted by
$$
e^{tD_{X,p}}f(x)=f(\exp tX\cdot x),\       t\in \mathbb{R}, \
     f\in L_{p}({\bf M}), \       x\in {\bf M}.
$$

According to the general theory of one-parameter groups in Banach
spaces \cite{BB}, Ch.\ I,  the operator $D_{X,p}$ is a closed
operator on every $L_{p}({\bf M}), 1\leq p\leq \infty$.   In order to
simplify notation, we will often write $D_{X}$ in place of
$D_{X, p}$.

If $\textbf{g}$ is the Lie algebra of a compact Lie group $G$ then
(\cite{H3}, Ch.\ II, Proposition 6.6,) it is a direct sum
$\textbf{g}=\textbf{a}+[\textbf{g},\textbf{g}]$, where
$\textbf{a}$ is the center of $\textbf{g}$, and
$[\textbf{g},\textbf{g}]$ is a semi-simple algebra. Let $Q$ be a
positive-definite quadratic form on $\textbf{g}$ which, on
$[\textbf{g},\textbf{g}]$, is opposite to the Killing form. Let
$X_{1},...,X_{d}$ be a basis of
$\textbf{g}$, which is orthonormal with respect to $Q$.
 Since the form $Q$ is $Ad(G)$-invariant, the operator
$$
-X_{1}^{2}-X_{2}^{2}-\    ... -X_{d}^{2},    \ d=dim\ G
$$
is a bi-invariant operator on $G$. This implies in particular that
the
   corresponding operator on $L_{p}({\bf M}), \      1\leq p\leq\infty,$
\begin{equation}
\mathcal{L}=-D_{1}^{2}- D_{2}^{2}- ...- D_{d}^{2}, \>\>\>
       D_{j}=D_{X_{j}}, \        d=dim \ G,\label{Laplacian}
\end{equation}
commutes with all operators $D_{j}=D_{X_{j}}$. This operator
$\mathcal{L}$, which is usually called the Laplace operator, is elliptic, and is
involved in most of the constructions and results of our paper.
However, as we discussed in the introduction, in many of the results 
prior to section 6,
one could use other second order elliptic differential operators.

In the rest of the paper, the notation $\mathbb{D}=\{D_{1},...,
D_{d}\},\>\>\> d=dim \ G,$ will be used for the differential operators
on $L_{p}({\bf M}), 1\leq p\leq \infty,$ which are involved in the
formula (\ref{Laplacian}).

In some situations the operator $\mathcal{L}$ is essentially the
Laplace-Beltrami operator ($-d^*d$) of an invariant metric on ${\bf M}$. This
happens for example in the following cases.

1) If ${\bf M}$ is a $d$-dimensional torus, and $-\mathcal{L}$ is the sum
of squares of partial derivatives.

2) If the manifold ${\bf M}$ is itself a group $G$ which is compact and
semi-simple, then $\mathcal{L}$ is exactly the Laplace-Beltrami
operator of an invariant metric on $G$ (\cite{H2}, Ch. II,
Exercise A4).

3) If ${\bf M}=G/K$ is a compact symmetric space of
  rank one, then the operator
 $\mathcal{L}$ is proportional to the Laplace-Beltrami operator
of an invariant metric on $G/K$. This follows from the fact that, in
the rank one case, every second-order operator  which commutes with
all isometries $x\rightarrow g\cdot x, \>\>\>x\in  {\bf M},\>\>\> g\in
G,$ is proportional to the Laplace-Beltrami operator (\cite{H2},
Ch. II, Theorem 4.11).

Let us stress one more time that in the present paper we  use only
the properties that the operator $\mathcal{L}$ has the form (\ref{Laplacian}) and commutes with all
isometries $g:{\bf M}\rightarrow g\cdot {\bf M},\>\>\> g\in G,\>\>\> $ of ${\bf M}$,
and we do not explore its relation to the Laplace-Beltrami
operator of the invariant metric.

Note that if ${\bf M}=G/K$ is a compact symmetric space, then the number
$d=dim G$ of operators in the formula (\ref{Laplacian}) can be
strictly larger than the dimension $ n=dim {\bf M}$. For example, on a
two-dimensional sphere $\mathbb{S}^{2}$ the Laplace-Beltrami
operator $L_{\mathbb{S}^{2}}$ can be written as
\begin{equation}
\mathcal{L}_{\mathbb{S}^{2}}= -(D_{1}^{2}+ D_{2}^{2}+
D_{3}^{2}),\label{S-Laplacian}
\end{equation}
where $D_{i}, i=1,2,3,$ generates a rotation  in $\mathbb{R}^{3}$
around the coordinate axis $x_{i}$:
\begin{equation}
D_{i}=x_{j}\partial_{k}-x_{k}\partial_{j},
\end{equation}
where $j,k\neq i.$

\section{Function spaces on compact homogeneous manifolds}

 The  operator $\mathcal{L}$  is an elliptic differential operator which
is defined on $C^{\infty}({\bf M})$, and we will use the same notation
$\mathcal{L}$ for its closure from $C^{\infty}({\bf M})$ in $L_{p}({\bf M}),
1\leq p\leq \infty$. In the case $p=2$ this closure is a
self-adjoint positive definite operator on the space $L_{2}({\bf M})$.
The spectrum of this operator is discrete and goes to infinity
$0=\lambda_{0}<\lambda_{1}\leq \lambda_{2}\leq ...$ . Let
$u_{0}, u_{1}, u_{2}, ...$ be a corresponding
complete system of real-valued orthonormal eigenfunctions, and let
$\textbf{E}_{\omega}(\mathcal{L}),\ \omega>0,$ be the span of all
eigenfunctions of $\mathcal{L}$, whose corresponding eigenvalues
are not greater than $\omega$.

We say that a  function $f\in L_{p}(M), 1\leq p\leq \infty,$
belongs to the Bernstein space
$\mathbf{B}_{\omega}^{p}(\mathbb{D}),\
\mathbb{D}=\{D_{1},...,D_{d}\},\ d=dim G,$ if and only if for every
$1\leq i_{1},...i_{k}\leq d$, the following Bernstein inequality
holds:
 \begin{equation}
 \|D_{i_{1}}...D_{i_{k}}f\|_{p}\leq
 \omega^{k}\|f\|_{p},\ k\in \mathbb{N}, \  1\leq p\leq \infty.\label{Bern}
\end{equation}

We say that a  function $f\in L_{2}(M)$
belongs to the Bernstein space
$\mathbf{B}_{\omega}^{2}(\mathcal{L}), $ if and only if for every
$k\in \mathbb{N}$, the following Bernstein inequality holds:
$$
\|\mathcal{L}^{k}f\|_{2}\leq \omega^{2k}\|f\|_{2},\ k\in \mathbb{N}.
$$

Since $\mathcal{L}$ on the space $L_{2}(M)$ is self-adjoint and
positive-definite, there exists a unique positive square root
$\mathcal{L}^{1/2}$. Thus the last inequality is
equivalent to the inequality
$$
\|\mathcal{L}^{k/2}f\|_{2}\leq \omega^{k}\|f\|_{2},\ k\in
\mathbb{N}.
$$
It was shown in \cite{Pes08} that the Bernstein spaces
$\mathbf{B}_{\omega}^{p}(\mathbb{D}),
\mathbf{B}_{\omega}^{p}(\mathcal{L})$ are linear spaces. Moreover, it 
was shown in the same paper that the following equality holds:
$$
\mathbf{B}_{\omega}^{p}(\mathbb{D})
=\mathbf{B}_{\omega}^{q}(\mathbb{D})\equiv
\mathbf{B}_{\omega}(\mathbb{D}),\ \mathbb{D}=\{D_{1},...,D_{d}\},\
d=dim G,
$$
which means that if the Bernstein-type inequalities (\ref{Bern})
are satisfied for a single  $1\leq p\leq \infty$, then they are
satisfied for all $1\leq p\leq \infty$.

 The 
 following embeddings were also proved in \cite{Pes08} which describe relations between
Bernstein spaces $
\textbf{B}_{\omega}(\mathbb{D}),\ \mathbb{D}=\{D_{1},...,D_{d}\},\  d=dim
G,$ and the spaces $\textbf{E}_{\lambda}(\mathcal{L})$ for $
-\mathcal{L}=D_{1}^{2}+ D_{2}^{2}+ ...+ D_{d}^{2},\  d=dim G$:
\begin{equation}
\textbf{E}_{\omega}(\mathcal{L})\subset
\textbf{B}_{\sqrt{\omega}}(\mathbb{D}),\  d=dim G,\ \omega>0.\label{A}
\end{equation}

\begin{equation}
\textbf{B}_{\omega}(\mathbb{D})\subset\textbf{E}_{\omega^{2}d}(\mathcal{L})\subset
\textbf{B}_{\omega\sqrt{d}}(\mathbb{D}),\  d=dim G,\ \omega>0.\label{B}
\end{equation}
These embeddings obviously  imply the equality
$$
\bigcup_{\omega>0} \textbf{B}_{\omega}(\mathbb{D})=\bigcup_{j}
\textbf{E}_{\lambda_{j}}(\mathcal{L}),
$$
which means  \textit{that a function on ${\bf M}$ satisfies a Bernstein
inequality (\ref{Bern}) in a norm of $L_{p}(M), 1\leq p\leq
\infty,$ if and only if it is a linear combination of
eigenfunctions of $\mathcal{L}$.}
As a consequence we have the following Bernstein-Nikolski inequality:
 for every $\varphi \in \textbf{E}_{\omega}(\mathcal{L})$ and 
 $1\leq p\leq q\leq\infty,
$
\begin{equation}
\|\mathcal{L}^{k}\varphi\|_{q}\leq C(M)
\omega^{2k+\frac{n}{p}-\frac{n}{q}}\|\varphi\|_{p}, k\in
\mathbb{N},\ n=dim {\bf M},\ d=dim G, \label{C}
\end{equation}
for a certain constant $C({\bf M})$ which depends only on the manifold.

It is known  (\cite{Z}, Ch. IV)   that every compact Lie
group can be considered to be a closed subgroup of the orthogonal
group $O(\mathbb{R}^{N})$ of some Euclidean space
$\mathbb{R}^{N}$.  For a compact symmetric space $M=G/K$, 
where $G$ is a compact
Lie group, we can identify $\bold{M}$ with the orbit
of a unit vector $v\in \mathbb{R}^{N}$ under the action of a subgroup
of the orthogonal group $O(\mathbb{R}^{N})$ in $\mathbb{R}^{N}$.
In this case $K$ will be  the stationary group of $v$. Such an
embedding of $\bold{M}$ into $\mathbb{R}^{N}$ is called equivariant.

We choose an orthonormal  basis in $\mathbb{R}^{N}$ for which the
first vector is the vector $v$: $e_{1}=v, e_{2},...,e_{N}$. Let $
\textbf{P}_{m}(\bold{M}) $ be the space of restrictions to $M$ of all
polynomials in $\mathbb{R}^{N}$ of degree $m$. This space is
closed in the norm of $L_{p}(\bold{M}), 1\leq p\leq \infty,$ which is
constructed with respect to the $G$-invariant measure on $\bold{M}$.

Let $T$ be the quasi-regular representation of $G$ in the space
$L_{p}(\bold{M}),\ 1\leq p\leq \infty$ . In other words, if $ f\in
L_{p}(\bold{M}),\ g\in G,\ x\in \bold{M}$, then
\begin{equation}
\label{quasi}
\left(T(g)f\right)(x)=f(g^{-1}x).
\end{equation}
The Lie algebra $\textbf{g}$ of the group $G$ is formed by those $N\times N$
skew-symmetric matrices $X$ for which $\exp tX\in  G$ for all
$t\in \mathbb{R}$. The scalar product in $\textbf{g}$ is given by
the formula
$$
<X_{1},X_{2}>=\frac{1}{2}tr(X_{1}X_{2}^{t})=-\frac{1}{2}tr(X_{1}X_{2}),\ 
X_{1},X_{2}\in \textbf{g}.
$$
Let $X_{1},X_{2},...,X_{d}$ be an orthonormal basis of
$\textbf{g},\ \dim \textbf{g}=d,$ and $D_{1}, D_{2}, ...,D_{d}$ be
the corresponding infinitesimal operators of the quasi-regular
representation of $G$ in $L_{p}(\bold{M}), 1\leq p\leq\infty$.

The following relations were proved in \cite{Pes08}:

\begin{equation}
\textbf{P}_{m}(\bold{M})\subset \textbf{B}_{m}(\mathbb{D})\subset
\textbf{E}_{m^{2}d}(\mathcal{L})\subset
\textbf{B}_{m\sqrt{d}}(\mathbb{D}),\ d=dim G,\ m\in \mathbb{N},\label{DD}
\footnote{We would like to point out that some of the indices in our formulas  (\ref{B}),  (\ref{DD})
are  different from the indices in the  corresponding formulas in \cite{Pes08}.}
\end{equation}
and
\begin{equation}
\bigcup _{m}\textbf{P}_{m}(\bold{M})=\bigcup_{\omega}
\textbf{B}_{\omega}(\mathbb{D})=\bigcup_{j}
\textbf{E}_{\lambda_{j}}(\mathcal{L}),\ m\in \mathbb{N,}
\end{equation}
where $ \textbf{P}_{m}(\bold{M})$ is the space of restrictions to $\bold{M}$ of polynomials of degree at most $m$.

Let $B(x,r)$ be a metric ball on ${\bf M}$ whose center is $x$ and
radius is $r$. The following important Lemma can be found  in  \cite{Pes00},
\cite{Pes04a}.

\begin{lem}
For any Riemannian manifold of bounded geometry ${\bf M}$  there exists
a natural number $N_{{\bf M}}$, such that  for any sufficiently small $\rho>0$
there exists a set of points $\{y_{\nu}\}$ such that:
\begin{enumerate}
\item the balls $B(y_{\nu}, \rho/4)$ are disjoint,

\item  the balls $B(y_{\nu}, \rho/2)$ form a cover of ${\bf M}$,

\item  the multiplicity of the cover by balls $B(y_{\nu}, \rho)$
is not greater than $N_{{\bf M}}.$
\end{enumerate}\label{covL}
\end{lem}

\begin{defn}
Any set of points $M_{\rho}=\{y_{\nu}\}$ which is as described in
Lemma \ref{covL} will be called a metric
$\rho$-lattice.\label{D1}
\end{defn}

To define Sobolev spaces,  we fix a  cover $B=\{B(y_{\nu}, r_{0})\}$ of $\bold{M}$ of finite
multiplicity $N(M)$ (see Lemma \ref{covL})
\begin{equation}
\bold{M}=\bigcup B(y_{\nu}, r_{0}),\label{cover}
\end{equation}
where $B(y_{\nu}, r_{0})$ is a  ball centered at $y_{\nu}\in \bold{M}$ of radius
$r_{0}\leq \rho_{\bold{M}},$ contained in a coordinate chart, and consider a fixed partition of unity
$\Psi=\{\psi_{\nu}\}$ subordinate to this cover. The Sobolev
spaces $W^{k}_{p}(\bold{M}), k\in \mathbb{N}, 1\leq p<\infty,$ are
introduced as the completion of $C^{\infty}(\bold{M})$ with respect
to the norm
\begin{equation}
\|f\|_{W^{k}_{p}(\bold{M})}=\left(\sum_{\nu}\|\psi_{\nu}f\|^{p}
_{W^{k}_{p}(B(y_{\nu}, r_{0}))}\right) ^{1/p}.\label{Sobnorm}
\end{equation}
Any two such norms are equivalent.

Now we turn to Besov spaces on $\bold{M}$.
Suppose that
$-\infty < \alpha < \infty$ and $0 < p,q \leq \infty$. 
We use the notation for inhomogeneous Besov spaces $B_p^{\alpha q}$ on $\bf{R}^{n}$ from \cite{FJ1}.  
Thus, on $\bf{R}^{n}$, one takes
any $\Phi \in {\mathcal S}$ supported in the closed unit ball, which does not vanish anywhere in the ball of
radius $5/6$ centered at $0$.  One also takes functions $\varphi_{\nu} \in {\mathcal S}$ for $\nu \geq 1$, 
supported in the annulus $\{\xi: 2^{\nu-1} \leq |\xi| \leq 2^{\nu+1}\}$, satisfying 
$|\varphi_{\nu}(\xi)| \geq c > 0$ for $3/5 \leq 2^{-\nu}|\xi| \leq 5/3$ and also 
$|\partial^{\gamma} \varphi_{\nu}| \leq c_{\gamma} 2^{-\nu \gamma}$ for every multiindex $\gamma$.
The Besov space $B_p^{\alpha q}(\bf{R}^{n})$ is then the space of $F \in {\mathcal S}'(\bf{R}^{n})$ such that

\begin{equation}
\label{besrndf}
\|F\|_{B_p^{\alpha q}} = \|\check{\Phi}*F\|_{L_p} + 
\left(\sum_{\nu = 0}^{\infty} (2^{\nu \alpha}\|\check{\varphi}_{\nu}*F\|_{L_p})^q \right)^{1/q} < \infty.
\end{equation}

(Here we use the usual conventions if $p$ or $q$ is $\infty$.
The definition of $B_p^{\alpha q}(\bf{R}^{n})$ is independent of the choices of 
$\Phi, \varphi_{\nu}$ (\cite{Peet}, page 49).  Moreover, $B_p^{\alpha q}(\bf{R}^{n})$ is a quasi-Banach
space, and the inclusion $B_p^{\alpha q} \subseteq {\mathcal S}'$ is continuous
(\cite{Trieb}, page 48).  In particular the space $B^{\alpha}_{\infty,\infty}(\bf{R}^{n}) =  {\mathcal C}^{\alpha}(\bf{R}^{n})$,
which is the usual H\"older space if $0 < \alpha < 1$, or in general a H\"older-Zygmund space for $\alpha > 0$ 
(\cite{Trieb}, page 51).  It is not hard to see, by using the definition and the Fourier transform, that
if $K \subseteq \bf{R}^{n}$ is compact, and if $N$ is sufficiently large, then
\begin{equation}
\label{ccNbes}
\{F \in C^N: \mbox{ supp}F \subseteq K\} \subseteq B_p^{\alpha q}
\end{equation}
where the inclusion map is continuous if we regard the left side as a subspace of $C^N$.

If $\eta: \bf{R}^{n} \rightarrow \bf{R}^{n}$ is a diffeomorphism which equals
the identity outside a compact set, then one can define $F \circ \eta$ for $F \in B_p^{\alpha q}(\bf{R}^{n})$,
and the map $F \rightarrow F \circ \eta$ is bounded on the Besov spaces (\cite{Trieb}, chapter 2.10).
These facts then enable one
to define $B_p^{\alpha q}({\bf M})$: let $(W_i, \chi_i)$ be a finite atlas on ${\bf M}$ with charts $\chi_i$ 
mapping $W_i$ into the unit ball on $\bf{R}^{n}$, and suppose $\{\zeta_i\}$ is a partition of unity subordinate
to the $W_i$.  Then one defines $B_p^{\alpha q}({\bf M})$ to be the space of distributions $F$ on 
${\bf M}$ for which 

\[ \|F\|_{B_p^{\alpha q}({\bf M})} = \sum_i \|(\zeta_i F) \circ \chi_i^{-1}\|_{B_p^{\alpha q}({\bf R}^n)} < \infty. \]

This definition does not depend on the choice of charts or partition of unity (\cite{Triebman}).\\

It was also shown in  \cite{Trieb1} , \cite{Triebman}  that the Besov space $B_p^{\alpha q}({\bf M})$ is exactly the interpolation space
$(L_{p}(\bold{M}),W^{r}_{p}(\bold{M}))^{K}_{\alpha/r,q},\ 0<\alpha<r\in \mathbb{N},\
1\leq p, q\leq \infty,$ where $K$ is the Peetre interpolation
functor.

Now we are going to describe Besov spaces in different terms. We consider the system of
vector fields $\mathbb{D}=\{D_{1},...,D_{d}\},\ d=dim G, $\ on
$\bold {M} =G/K,$ which was described above. Since the vector
fields $\mathbb{D}=\{D_{1},...,D_{d}\}$ generate the tangent space
at every point of $\bf M$, and $\bf M$ is compact, it is clear that the
Sobolev norm (\ref{Sobnorm}) is equivalent to the norm
\begin{equation}
\|f\|_{p}+\sum_{j=1}^{k} \sum_{1\leq i_{1},...,i_{j}\leq
d}\|D_{i_{1}}...D_{i_{j}}f\|_{p},\ 1\leq p\leq
\infty.\label{mixednorm}
\end{equation}
 Using the closed graph theorem and
the fact that each $D_{i}$ is a closed operator in $L_{p}({\bf M}),\ 
1\leq p \leq \infty$,\ it is easy to show that the norm
(\ref{mixednorm}) is equivalent to the norm
\begin{equation}
|||f|||_{k,p}=\|f\|_{p}+\sum_{1\leq i_{1},..., i_{k}\leq
d}\|D_{i_{1}}...D_{i_{k}}f\|_{p},\ 1\leq p\leq
\infty.\label{mixednorm2}
\end{equation}

For the same operators as above ($D_{1},...,D_{d},\ d=dim \ G$),  let $T_{1},..., T_{d}$
be the corresponding one-parameter groups of translation along integral
curves of the corresponding vector  fields i.e.
 \begin{equation}
 T_{j}(\tau)f(x)=f(\exp \tau X_{j}\cdot x),
 x\in \bold{M}, \tau \in \mathbb{R}, f\in L_{2}(\bold{M});
 \end{equation}
 here $\exp \tau X_{j}\cdot x$ is the integral curve of the vector field
 $X_{j}$ which passes through the point $x\in \bold{M}$.
 The modulus of continuity is introduced as
\begin{equation}
\Omega_{p}^{r}( s, f)= $$ $$\sum_{1\leq j_{1},...,j_{r}\leq
d}\sup_{0\leq\tau_{j_{1}}\leq s}...\sup_{0\leq\tau_{j_{r}}\leq
s}\|
\left(T_{j_{1}}(\tau_{j_{1}})-I\right)...\left(T_{j_{r}}(\tau_{j_{r}})-I\right)f\|_{L_{p}(\bold{M})},\label{M}
\end{equation}
where $f\in L_{p}(\bold{M}),\ r\in \mathbb{N},  $ and $I$ is the
identity operator in $L_{p}(\bold{M}).$   We consider the space of all functions in $L_{p}(\bold{M})$ for which the
following norm is finite:
\begin{equation}
\|f\|_{L_{p}(\bold{M})}+\left(\int_{0}^{\infty}(s^{-\alpha}\Omega_{p}^{r}(s,
f))^{q} \frac{ds}{s}\right)^{1/q} , 1\leq p,q<\infty,\label{BnormX}
\end{equation}
with the usual modifications for $q=\infty$.

 The following Theorem follows from general results of the second author about interpolation
 in spaces of representations of Lie groups  \cite{Pes79}-\cite{Pes88}: 
 
 \begin{thm}
 
 The norm of the Besov space $B_p^{\alpha q}({\bf M})=(L_{p}(\bold{M}),W^{r}_{p}(\bold{M}))^{K}_{\alpha/r,q},\ 
 0<\alpha<r\in \mathbb{N},\ 
1\leq p, q\leq \infty,$ is equivalent to the norm (\ref{BnormX}). Moreover, the norm
(\ref{BnormX}) is equivalent to the norm
\begin{equation}
\|f\|_{W_{p}^{[\alpha]}(\bold{M})}+\sum_{1\leq j_{1},...,j_{[\alpha] }\leq d}
\left(\int_{0}^{\infty}\left(s^{[\alpha]-\alpha}\Omega_{p}^{1}
(s,D_{j_{1}}...D_{j_{[\alpha]}}f)\right)^{q}\frac{ds}{s}\right)^{1/q}\label{nonint}
\end{equation}
if $\alpha$ is not integer ($[\alpha]$ is its integer part).  If
$\alpha=k\in \mathbb{N}$ is an integer then the norm
(\ref{BnormX}) is equivalent to the norm (Zygmund condition)
\begin{equation}
\|f\|_{W_{p}^{k-1}(\bold{M})}+ \sum_{1\leq j_{1}, ... ,j_{k-1}\leq d }
\left(\int_{0}^{\infty}\left(s^{-1}\Omega_{p}^{2}(s,
D_{j_{1}}...D_{j_{k-1}}f)\right)
 ^{q}\frac{ds}{s}\right)^{1/q}.\label{integer}
\end{equation}
\end{thm}

When $p=2$ the first of these norms can be changed to

\begin{equation}
\|f\|_{H^{[\alpha]}(\bold{M})}+\sum_{1\leq j_{1},...,j_{[\alpha] }\leq d}
\left(\int_{0}^{\infty}\left(s^{[\alpha]-\alpha}\Omega_{2}^{1}
(s,      \mathcal{L}^{[\alpha]/2}       f)\right)^{q}\frac{ds}{s}\right)^{1/q}
\end{equation}
and the second to

\begin{equation}
\|f\|_{H^{k-1}(\bold{M})}+ \sum_{1\leq j_{1}, ... ,j_{k-1}\leq d }
\left(\int_{0}^{\infty}\left(s^{-1}\Omega_{2}^{2}(s,
     \mathcal{L}^{(k-1)/2}      f)\right)
 ^{q}\frac{ds}{s}\right)^{1/q}.
\end{equation}
For a function $f\in L_{2}(\bold{M})$ we introduce a notion of best
approximation
\begin{equation}
\mathcal{E}(f,\omega)=\inf_{g\in
\textbf{E}_{\omega}(\mathcal{L})}\|f-g\|_{2}=\left(\sum_{\lambda_{j}\geq\omega}c_{j}(f)^{2}\right)^{1/2},
\end{equation}
where $c_{j}=\left<f,u_{\lambda_{j}}\right>$ are the Fourier coefficients of $f$.

A description of  Besov spaces  $B_{2}^{\alpha q}({\bf M}),\ \alpha>0,\ 1\leq q\leq\infty,$ in terms
of the best approximation $\mathcal{E}(f,\omega)$ was given in  \cite{Pes09}, Theorems 1.1 and 1.2.
We will obtain a generalization of these results  in Theorem \ref{ernrmeq} below.

\section{Plancherel-Polya ($=$Marcinkiewicz-Zygmund) inequalities}

In this section, we again consider a compact homogeneous Riemannian manifold ${\bf M}$, and the
elliptic self-adjoint positive definite operator $\mathcal{L}$ on
$L_{2}({\bf M})$, which was introduced in (\ref{Laplacian}).  However, 
the results of this section hold for general 
${\bf M}$ and ${\mathcal L}$ (if at least
${\bf M}$ is smooth, compact   and
${\mathcal L}$ is a positive elliptic self-adjoint second-order differential operator on ${\bf M}$).

 Since the
operator $\mathcal{L}$ is of order two, the dimension
$\mathcal{N}_{\omega}$ of the space ${\mathbf E}_{\omega}(\mathcal{L})$ is
given asymptotically by Weyl's formula \cite{Sogge93}
\begin{equation}
\mathcal{N}_{\omega}({\bf M})\asymp C({\bf M})\omega^{n/2},\label{W}
\end{equation}
where $n=dim {\bf M}$.

 The next two theorems were proved in \cite{Pes00}, \cite{Pes04b},
for a Laplace-Beltrami operator on a Riemannian manifold of
bounded geometry,  but their proofs go through for any
$C^{\infty}$-bounded uniformly elliptic self-adjoint positive
definite differential operator on ${\bf M}$. 
In what follows the notation $n=dim\  {\bf M}$ is used.

\begin{thm}  There exist constants $\>\>C_{1}=C_{1}({\bf M},\mathcal{L})>0\>\>$ and
 $\>\>\>
\rho_{0}({\bf M},\mathcal{L})>0,$ such that for any  natural number $m>n/2$,
any $0<\rho<\rho_{0}({\bf M},\mathcal{L})$,  and any $\rho$-lattice
$M_{\rho}=\{x_{k}\}$,  the following inequality holds:
$$\left(\sum _{x_{k}\in M_{\rho}}|f(x_{k})|^{2}\right)^{1/2}\leq
 C_{1}\rho^{-n/2}\|f\|_{H^{m}({\bf M})},
 $$\label{MT}
for all  $f\in H^{m}({\bf M}), \>\>\> m>n/2,\>\>\> m\in
\mathbb{N}.$\label{T1}
\end{thm}

\begin{thm} There exist constants
$C_{2}=C_{2}({\bf M},\mathcal{L})>0,$ and $\>\>\>
\rho_{0}({\bf M},\mathcal{L})>0,$ such that for any  natural $m>n/2$,
any $0<\rho<\rho_{0}({\bf M},\mathcal{L})$, and any $\rho$-lattice
$M_{\rho}=\{x_{k}\}$ the following  inequality holds

\begin{equation}
\|f\|_{H^{m}({\bf M})}\leq C_{2}\left\{\rho^{n/2}\left(\sum_{x_{k}\in M_{\rho}}
|f(x_{k})|^{2}\right)^{1/2}+\rho^{2m}\|\mathcal{L}^{m}f\|\right\},\>\>\>
m\in \mathbb{N},\>\>\> m>n/2.\label{rightPP}
\end{equation}
\end{thm}

Using  the constant $C_{2}({\bf M},\mathcal{L})$ from this Theorem, we
define another constant
\begin{equation}
c_{0}=c_{0}({\bf M},\mathcal{L})=\left(2C_{2}({\bf M},\mathcal{L})\right)^{-1/2m_{0}},\label{const}
\end{equation}
where $m_{0}=1+n/2,\>\>\>n=dim {\bf M}$.

 The previous  Theorem  and the Bernstein inequality imply the
following Plancherel-Polya-type inequalities.  Such
inequalities are also known as Marcinkewicz-Zygmund inequalities.

\begin{thm}
There exists a constant $\>\>\>c_{2}=c_{2}({\bf M},\>\>\mathcal{L})$
such that for any $\omega>0$, and for every metric $\rho$-lattice
$M_{\rho}=\{x_{k}\}$ with $\rho= c_{0}\omega^{-1/2}$, the following
inequalities hold:

\begin{equation}
\rho^{-n/2}\|f\|_{L_{2}({\bf M})} \leq c_{2}\left(\sum_{k}
|f(x_{k})|^{2}\right)^{1/2}\label{PP1}
\end{equation}
for all $f\in {\mathbf E}_{\omega}(\mathcal{L})$ and $c_{0}$ defined in
(\ref{const}). Moreover, for the same lattice  there exists a
constant $\>\> c_{1}=c_{1}({\bf M},\>\>\mathcal{L},\>\> \omega)\>\>$
such that
\begin{equation}
c_{1}\left(\sum_{k}|f(x_{k})|^{2}\right)^{1/2}\leq
\rho^{-n/2}\|f\|_{L_{2}({\bf M})}\label{PP2}
\end{equation}
\end{thm}
\begin{proof}Indeed, if $f\in {\mathbf E}_{\omega}({\bf M})$, and $\rho$ in (\ref{rightPP}) is given
by $\rho=c_{0}\omega^{-1/2}$ where $c_{0}$ was defined in
(\ref{const}), then by the Bernstein inequality
$$
C_{2}\rho^{2m_{0}}\|\mathcal{L}^{m_{0}} f\|\leq
C_{2}\left(\rho^{2}\omega^{-1}\right)^{m_{0}}\|f\|=\frac{1}{2}\|f\|.
$$
The inequality (\ref{rightPP}) now implies that
\begin{equation}
\rho^{-n/2}\|f\|_{2}\leq c_{2}\left(\sum_{x_{k}\in M_{\rho}}
|f(x_{k})|^{2}\right)^{1/2},
\end{equation}
with $c_{2}=2C_{2}({\bf M},\mathcal{L})$ and $C_{2}({\bf M},\mathcal{L})$ is
the same as in (\ref{rightPP}). To prove (\ref{PP2}) we apply
elliptic regularity of $\mathcal{L}$ to obtain
\begin{equation}
\|f\|_{H^{m}({\bf M})}\leq
C({\bf M},\mathcal{L})\left(\|f\|+\|\mathcal{L}^{m/2}f\|\right)
\end{equation}
and then the Bernstein inequality gives
\begin{equation}
\|f\|_{H^{m}({\bf M})}\leq C({\bf M},\mathcal{L})(1+\omega^{m/2})\|f\|.
\end{equation}
By choosing $m_{0}=1+n/2$ for $\>\>m\>\>$ and using Theorem
\ref{T1}, we obtain (\ref{PP2}) with
$$
c_{1}=\left\{C_{1}({\bf M},\mathcal{L})C({\bf M},\mathcal{L})\left(1+\omega^{m_{0}}/2\right)\right\}^{-1}.
$$
\end{proof}

\begin{col}

There exist constants $\>\>\>c_{1}=c_{1}({\bf M},\mathcal{L})>0,\>\>$
$c_{2}=c_{2}({\bf M},\mathcal{L})>0,$
 and
$\>\>c_{0}=c_{0}({\bf M},\mathcal{L})>0,$ such that for any $\omega>0$, and for
every metric $\rho$-lattice $M_{\rho}=\{x_{k}\}$ with $\rho=
c_{0}\omega^{-1/2}$, the following Plancherel-Polya inequalities
hold:

\begin{equation}
c_{1}\left(\sum_{k}|f(x_{k})|^{2}\right)^{1/2}\leq\rho^{-n/2}\|f\|_{L_{2}({\bf M})}
\leq c_{2}\left(\sum_{k} |f(x_{k})|^{2}\right)^{1/2}, \label{completePlPo100}
\end{equation}
for all $f\in {\mathbf E}_{\omega}(\mathcal{L})$ and $n=\dim \  {\bf M}$.
\label{completePlPo200}
\end{col}

The following Theorem shows that our lattices (appearing in the
previous Theorems) always produce sampling sets with essentially the
optimal number of sampling points (see also
\cite{Pes04a},\cite{Pes09}).

\begin{thm}
 If the  constant $c_{0}({\bf M},\mathcal{L})>0$ is the same
as above, then  for any $\omega>0$ and $\rho=c_{0}\omega^{-1/2}$,
there exist $C_{1}({\bf M},\mathcal{L}), C_{2}({\bf M},\mathcal{L})$ such
that the number of points in  any $\rho$-lattice $M_{\rho}$
satisfies the following inequalities
\begin{equation}
C_{1}\omega^{n/2}\leq |M_{\rho}|\leq
C_{2}\omega^{n/2};\label{rate}
\end{equation}

\end{thm}\label{FT}
\begin{proof}
 According to the
definition of a lattice $M_{\rho}$ we have
$$
|M_{\rho}|\inf_{x\in {\bf M}}Vol(B(x,\rho/4))\leq Vol({\bf M})\leq
|M_{\rho}|\sup_{x\in {\bf M}}Vol(B(x,\rho/2))
$$
or
$$
\frac{Vol({\bf M})}{\sup_{x\in {\bf M}}Vol(B(x,\rho/2))}\leq |M_{\rho}|\leq
\frac{Vol({\bf M})}{\inf_{x\in {\bf M}}Vol\left(B(x,\rho/4)\right)}.
$$
Since for  certain $c_{1}({\bf M}), c_{2}({\bf M})$, all $x\in {\bf M}$ and all
sufficiently small $\rho>0$, one has a double inequality
$$
c_{1}({\bf M})\rho^{n}\leq  Vol(B(x,\rho))\leq c_{2}({\bf M})\rho^{n},
$$
and since $\rho=c_{0}\omega^{-1/2}, $ we obtain that for certain
$C_{1}({\bf M},\mathcal{L}), C_{2}({\bf M},\mathcal{L})$ and all $\omega>0$
\begin{equation}
C_{1}\omega^{n/2}\leq |M_{\rho}|\leq C_{2}\omega^{n/2}.
\end{equation}
\end{proof}

Since the inequalities (\ref{rate}) are in an agreement with 
Weyl's formula (\ref{W}),  the Theorem shows that
if $\omega>0$ is large enough, every uniqueness set $M_{\rho}$ for
${\mathbf E}_{\omega}(\mathcal{L})$ contains essentially the "correct" number of
points.

\section{Cubature formulas}

Again we work on 
a compact homogeneous Riemannian manifold ${\bf M}$, and use the
operator ${\mathcal L}$ of 
(\ref{Laplacian}).  However, 
the results of this section hold for general 
${\bf M}$ and ${\mathcal L}$ (if at least
${\bf M}$ is smooth and compact, and
${\mathcal L}$ is a positive elliptic self-adjoint second-order differential operator on ${\bf M}$).

Corollary \ref{completePlPo200}
 shows that  if $\vartheta_{k}$ is
the orthogonal projection of the Dirac measure $\delta_{x_{k}}$ on
the space ${\mathbf E}_{\omega}(\mathcal{L})$ (in a Hilbert space
$H^{-n/2-\varepsilon}({\bf M}), \>\>\varepsilon >0$, which can be defined as the domain of the operator $\mathcal{L}^{-n/4-\varepsilon/2}$) then there exist
constants $c_{1}=c_{1}({\bf M},\mathcal{L}, \omega)>0,\>\>
c_{2}=c_{2}({\bf M}mat,\mathcal{L})>0,$ such that the following frame
inequality holds
\begin{equation}
c_{1}\left(\sum_{k}\left|\left<f,\vartheta_{k}\right>\right|^{2}\right)^{1/2}
\leq \rho^{-n/2}\|f\|_{L_{2}({\bf M})} \leq
c_{2}\left(\sum_{k}\left|\left<f,\vartheta_{k}\right>\right|^{2}\right)^{1/2}.
\end{equation}
for all $f\in {\mathbf E}_{\omega}(\mathcal{L})$.

Let $M_{\rho}=\{x_{k}\}, \ k=1,...,N(M_{\rho}),$ be a $\rho$-lattice on ${\bf M}$ (see Lemma \ref{covL}). 
We construct the Voronoi partition of   ${\bf M}$ associated to the set 
  $M_{\rho}=\{x_{k}\}, \ k=1,...,N(M_{\rho})$.   Elements of this partition will be denoted as 
$ \mathcal{M}_{k,\rho}$.  Let us recall that the distance from each point in $ \mathcal{M}_{j,\rho}$
to $x_j$ is less than or equal to its distance to any 
other point of the family  $M_{\rho}=\{x_{k}\}, \ k=1,...,N(M_{\rho})$.
Some properties of this cover of   ${\bf M}$ are summarized in the following Lemma. 
which follows easily from  the definitions.

\begin{lem}
\label{mkrho}
The sets $\mathcal{M}_{k,
\rho}, \  k=1,...,N(M_{\rho}),$ have the following properties:

1) they are measurable;

2) they are disjoint;

3) they form a cover of ${\bf M}$;

4) there exist positive $a_{1},\  a_{2}$, independent of $\rho$ and the lattice $M_{\rho}=\{x_{k}\}$, such that 

\begin{equation}
\label{mkrhoway}
a_{1}\rho^{n}\leq \mu\left(\mathcal{M}_{k,\rho}\right)\leq a_{2}\rho^{n}.
\end{equation}

\end{lem}
Our next goal is to prove the following fact.
\begin{thm} Say $\rho > 0$, and let $\left\{\mathcal{M}_{k,\rho}\right\}$ be the disjoint cover
 of ${\bf M}$ which is associated with a $\rho$-lattice $M_{\rho}$.  If $\rho$ is 
sufficiently small  then for any sufficiently large $K\in \mathbb{N}$ there exists a $C(K)>0$ such that for all  smooth functions $f$ the following inequality holds:
\begin{equation}
\left|\sum_{\nu}\sum_{x_{k}\in M_{\rho}}\psi_{\nu}f(x_{k})\  \mu \mathcal{M}_{k,\rho}-\int_{{\bf M}}f(x)dx\right|\leq
 C(K)\sum_{|\beta|=1}^{K}\rho^{n/2+|\beta|}\|(I+\mathcal{L})^{|\beta|/2}f\|_{2},\label{closeness}
\end{equation}
where $C(K)$ is independent of $\rho$ and the $\rho$-lattice $M_{\rho}$.
\end{thm}
\begin{proof}
 We start with the Taylor series
\begin{equation}
\psi_{\nu}f(y)-\psi_{\nu}f(x_{k})=\sum_{1\leq |\alpha| \leq m-1} \frac{1}{\alpha
!}\partial^{\alpha}(\psi_{\nu}f)(x_{k})(x_{k}-y) ^{\alpha}+\label{Taylor}
\end{equation}
$$
\sum_{|\alpha|=m}\frac{1}{\alpha !}\int_{0}^{\tau}t^{m-1}\partial
^{\alpha}\psi_{\nu}f(x_{k}+t\theta)\theta^{\alpha}dt,
$$
where $ f\in
C^{\infty}(\mathbb{R}^{d}), \   y\in B(x_{k},\rho/2),\  x=(x^{(1)},...,x^{(d)}),\  y=(y^{(1)},...,y^{(d)}),\  \alpha=(
\alpha_{1},...,\alpha_{d}),\\ 
(x-y)^{\alpha}=(x^{(1)}-y^{(1)})^{\alpha_{1}}...
(x^{(d)}-y^{(d)})^{\alpha_{d}},\  \tau=\|x-x_{i}\|,\
\theta=(x-x_{i})/ \tau.
$

We are going to use the following inequality, which easily fellows from Lemma 6.19 in \cite{Ad},
and which is essentially the Sobolev imbedding theorem:
\begin{equation}
|(\psi_{\nu}f)(x_{k})|\leq C_{n,m}\sum_{0\leq j \leq m}
\rho^{j-n/p}\|(\psi_{\nu}f)\| _{W^{j}_{p}(B(x_{k},\rho))},\  1\leq p\leq \infty,\>\>
\label{basicineq}
\end{equation}
where $ m>n/p,$ and the functions $\{\psi_{\nu}\}$  form the partition
of unity which we used to define the Sobolev norm in
(\ref{Sobnorm}).
Using (\ref{basicineq}) for $p=1$ we obtain that the following inequality

\begin{equation}\left |\sum_{1\leq|\alpha|\leq m-1} \frac{1}{\alpha
!}\partial^{\alpha}(\psi_{\nu}f)(x_{k})(x_{k}-y) ^{\alpha}\right |\leq 
\end{equation}
$$
C(n,m)\rho^{|\alpha|}\sum_{1\leq|\alpha|\leq m} \sum_{0\leq |\gamma|\leq  m}\rho^{|\gamma|-n}\|\partial
^{\alpha+\gamma}(\psi_{\nu}f)\|_{L_{1}(B(x_{k},\rho))},\ \ m>n,\label{interm}
$$
 for some $C(n,m)\geq 0$. Since, by the Schwarz inequality,
 \begin{equation}
 \|\partial
^{\alpha}(\psi_{\nu}f)\|_{L_{1}(B(x_{k},\rho))}\leq C(n)\rho^{n/2}
 \|\partial
^{\alpha}(\psi_{\nu}f)\|_{L_{2}(B(x_{k},\rho))}
 \end{equation}
 we obtain the following estimate, which holds for small  $\rho$:
 \begin{equation}
\sup_{y\in B(x_{k},\rho)}\left | \sum_{1\leq|\alpha|\leq m-1} \frac{1}{\alpha
!}\partial^{\alpha}(\psi_{\nu}f)(x_{k})(x_{k}-y) ^{\alpha}\right |\leq 
\end{equation}
$$
C(n,m)\sum_{1\leq|\beta|\leq 2m}\rho^{|\beta|-n/2} \|\partial
^{\beta}(\psi_{\nu}f)\|_{L_{2}(B(x_{k},\rho))},\  m>n.
$$
Next, using the Schwarz inequality and the
assumption that
 $m>n=dim\ {\bf M},\  |\alpha|=m,$ we obtain
$$
\left |\int_{0}^{\tau}t^{m-1}\partial
^{\alpha}\psi_{\nu}f(x_{k}+t\theta)\theta^{\alpha}dt\right |\leq
$$
$$\int_{0}^{\tau}t^{m-n/2-1/2}|t^{n/2-1/2}\partial^{\alpha}
\psi_{\nu}f(x_{k}+t\theta)|dt\leq
$$
$$C\left(\int_{0}^{\tau}t^{2m-n-1}\right)^{1/2}\left(\int_{0}
^{\tau} t^{n-1}|\partial^{\alpha}\psi_{\nu}f(x_{k}+t\theta)| ^{2}dt\right)
^{1/2}\leq
$$
$$C\tau^{m-n/2}\left(\int_{0}^{\tau}t^{n-1}|\partial^{\alpha}
\psi_{\nu}f(x_{k}+t\theta)|^{2}dt\right)^{1/2}, \ m>n.
$$
We square this inequality, and integrate both sides of it over the ball
$B(x_{k},\rho/2)$, using the spherical coordinate system
$(\tau, \theta).$  We find

$$
\int_{B(x_{k},\rho)}\left |\int_{0}^{\tau}t^{m-1}\partial
^{\alpha}\psi_{\nu}f(x_{k}+t\theta)\theta^{\alpha}dt\right |^{2}\tau^{n-1}d\theta
d\tau\leq
$$
$$
C(m,n)\int_{0}^{\rho/2}\tau^{2m-n}\int_{0}^{2\pi}
\left |\int_{0}^{\tau}t^{n-1}\partial
^{\alpha}(\psi_{\nu}f)(x_{k}+t\theta)\theta^{\alpha}dt\right |^{2}\tau^{n-1}d\theta
d\tau\leq
$$
$$C(m,n)\int_{0}^{\rho/2}t^{n-1}\left(\int_{0}^{2\pi}\int_{0}^{\rho/2}
\tau^{2m-n}\left |\partial^{\alpha}(\psi_{\nu}f)(x_{k}+t\theta)\right |^{2}
\tau^{n-1}d\tau
d\theta\right)dt\leq
$$
$$
C_{m,n}\rho^{2|\alpha|}\|\partial^{\alpha}
(\psi_{\nu}f)\|^{2}_{L_{2}(B(x_{k},\rho))},
$$
where $\tau=\|x-x_{k}\|\leq\rho/2, \  m=|\alpha|>n.$ 
Let $\left\{\mathcal{M}_{k,\rho}\right\}$ be the Voronoi  cover of ${\bf M}$ which is associated 
with a $\rho$-lattice $M_{\rho}$ (see Lemma \ref{mkrho}). 
From here we obtain
\begin{equation}
\int _{\mathcal{M}_{k}}\left | \psi_{\nu}f(y)-\psi_{\nu}f(x_{k})\right |dx\leq  
\end{equation}
$$
C(n,m)\sum_{1\leq|\beta|\leq 2m}\rho^{|\beta|+n/2} \|\partial
^{\beta}(\psi_{\nu}f)\|_{L_{2}(B(x_{k},\rho))}
+
$$
$$
\sum_{|\alpha|=m}\frac{1}{\alpha !}\int _{B(x_{k}, \rho)}  \left  |\int_{0}^{\tau}t^{m-1}\partial
^{\alpha}\psi_{\nu}f(x_{k}+t\theta)\theta^{\alpha}dt \right |\leq 
$$
$$
C(n,m)\sum_{1\leq|\beta|\leq 2m}\rho^{|\beta|+n/2} \|\partial
^{\beta}(\psi_{\nu}f)\|_{L_{2}(B(x_{k},\rho))}+
$$
$$
\rho^{n/2}
\sum_{|\alpha|=m}\frac{1}{\alpha !}\left (\int _{B(x_{k}, \rho)}  \left  |\int_{0}^{\tau}t^{m-1}\partial
^{\alpha}\psi_{\nu}f(x_{k}+t\theta)\theta^{\alpha}dt \right |^{2}\tau^{n-1}d\tau d\theta\right)^{1/2}\leq 
$$
$$
C(n,m)\sum_{1\leq|\beta|\leq 2m}\rho^{|\beta|+n/2} \|\partial
^{\beta}(\psi_{\nu}f)\|_{L_{2}(B(x_{k},\rho))}.
$$

Next, we have the following inequalities
$$
\sum_{\nu}\sum_{x_{k}\in M_{\rho}}\psi_{\nu}f(x_{k})\  \mu \mathcal{M}_{k,\rho}-\int_{{\bf M}}f(x)dx=
$$
$$
-\sum_{\nu}\left(\sum_{k} \int _{\mathcal{M}_{k,\rho}}\psi_{\nu}f(x)dx-\sum_{k}\psi_{\nu}f(x_{k})\  \mu \mathcal{M}_{k,\rho} \right )\leq
$$
\begin{equation}
\sum_{\nu}\sum_{k}\left | \int _{\mathcal{M}_{k,\rho}}\psi_{\nu}f(x)-\psi_{\nu}f(x_{k})\  \mu \mathcal{M}_{k,\rho}dx \right |
\end{equation}
$$
\leq C(n,m)\rho^{n/2}\sum_{\nu}\sum_{x_{k}\in M_{\rho}}\sum_{1\leq |\beta|\leq 2m}\rho^{|\beta|} \|\partial
^{\beta}(\psi_{\nu}f)\|_{L_{2}(B(x_{k},\rho))},
$$
where $\ m>n.  $
Using the definition of the Sobolev norm and elliptic regularity of the operator $I+\mathcal{L}$, 
where $I$ is the identity operator on $L_{2}({\bf M})$,  we obtain the inequality (\ref{closeness}).
\end{proof}

Now we are going to prove existence of cubature formulas which are exact on $ {\mathbf E}_{\omega}({\bf M})$,
and have positive coefficients of the "right" size.

\begin{thm} 
\label{cubformula}
There exists  a  positive constant $a_{0}$,    such  that if  $\rho=a_{0}(\omega+1)^{-1/2}$, then
for any $\rho$-lattice $M_{\rho}$, there exist strictly positive coefficients $\lambda_{x_{k}}>0, 
 \  x_{k}\in M_{\rho}$, \  for which the following equality holds for all functions in $ {\mathbf E}_{\omega}({\bf M})$:
\begin{equation}
\label{cubway}
\int_{{\bf M}}fdx=\sum_{x_{k}\in M_{\rho}}\lambda_{x_{k}}f(x_{k}).
\end{equation}
Moreover, there exists constants  $\  c_{1}, \  c_{2}, $  such that  the following inequalities hold:
\begin{equation}
c_{1}\rho^{n}\leq \lambda_{x_{k}}\leq c_{2}\rho^{n}, \ n=dim\   {\bf M}.
\end{equation}
\end{thm}
\begin{proof}
By using the Bernstein inequality, and our Plancherel-Polya inequalities  (\ref{completePlPo100}), and assuming that 
\begin{equation}
\rho<\frac{1}{2\sqrt{\omega+1}}
\end{equation}
we obtain from (\ref{closeness}) the following inequality:
\begin{equation}
\label{clos-band}
\left|\sum_{\nu}\sum_{x_{k}\in M_{\rho}}\psi_{\nu}f(x_{k})\  \mu \mathcal{M}_{k,\rho}-\int_{{\bf M}}f(x)dx\right|\leq
C_{1}\rho^{n/2}\sum_{|\beta|=1}^{K}\left(\rho\sqrt{1+\omega }\right)^{|\beta|}\|f\|_{2}\leq 
$$
$$
C_{2}\rho^{n}\left(\rho\sqrt{1+\omega}\right)\left(\sum_{x_{k}\in M_{\rho}} |f(x_{k})|^{2}\right)^{1/2},
\end{equation}
where $C_{2}$ is independent of  $\rho\in \left(0, (2\sqrt{\omega+1}\right)^{-1})$
and the $\rho$-lattice $M_{\rho}$.

Let $R_{\omega}(\mathcal L)$ denote the space of real-valued functions in 
${\mathbf E}_{\omega}(\mathcal L)$.  Since the eigenfunctions of ${\mathcal L}$ may be taken to be real,
we have ${\mathbf E}_{\omega}(\mathcal L) = R_{\omega}(\mathcal L) + iR_{\omega}(\mathcal L)$, so 
it is enough to show that (\ref{cubway}) holds for all $f \in R_{\omega}(\mathcal L)$.

Consider the sampling operator
$$
S: f\rightarrow \{f(x_{k})\}_{x_{k}\in M_{\rho}},
$$
which maps  $R_{\omega}(\mathcal L)$ into the space  $\mathbb{R}^{|{\mathcal M_{\rho}}|}$
with the $\ell^2$ norm.
Let $V = S(R_{\omega}(\mathcal L))$ be the image of  $R_{\omega}(\mathcal L)$ under $S$.
$V$ is a  subspace of $\mathbb{R}^{|{\mathcal M_{\rho}}|}$,
and we consider it with the induced $\ell^2$ norm.
If $u \in V$, denote the linear functional $y \to (y,u)$ on $V$ by $\ell_u$.
By our  Plancherel-Polya inequalities (\ref{completePlPo100})
,  the map 
$$
\{f(x_k)\}_{x_{k}\in M_{\rho}} \to \int _{\bf M}fdx
$$ 
is a well-defined linear functional on the finite dimensional space
$V$, and so equals $\ell_v$ for some $v \in V$, which may
or may not  have all components positive.  On the other hand, if $w$ is the vector with components $\{\mu({\mathcal M}_{k,\rho})\}, \ x_{k}\in M_{\rho}$, then $w$ might
not be in $V$, but it has all components positive and of the right size
$$
a_{1}\rho^{n}\leq \mu\left(\mathcal{M}_{k,\rho}\right)\leq a_{2}\rho^{n},
$$
for some  positive $a_{1},\  a_{2}$, independent  of $\rho$ and the lattice $M_{\rho}=\{x_{k}\}$.
Since, for any vector $u \in V$ the norm of $u$  is exactly the norm of the corresponding functional 
$\ell_u$,  inequality (\ref{clos-band})
tells us that 
\begin{equation}
\label{2}
\|Pw-v\| \leq \|w-v\| \leq C_2\rho^n \left(\rho\sqrt{1+\omega}\right),
\end{equation}
where  $P$ is the orthogonal projection onto $V$. Accordingly, if $z$ is the real vector $v-Pw$, then
\begin{equation}
\label{3}
v+(I-P)w = w + z ,
\end{equation}
where $\|z\| \leq C_2\rho^n \left(\rho\sqrt{1+\omega}\right)$.  Note, that all components of the vector $w$ 
 are of order $O(\rho^{n})$, while the order of $\|z\|$ is  $O(\rho^{n+1})$. Accordingly, if 
 $\rho\sqrt{1+\omega}$ is sufficiently small, then 
$\lambda := w + z$ has all components positive and of the right size.  Since $\lambda = v + (I-P)w$, the linear
functional $y \to (y,\lambda)$ on $V$ equals $\ell_v$.  In other words, if the vector $\lambda$ has components
 $\{\lambda_{x_{k}}\}, \ x_{k}\in M_{\rho},$ then 
$$
\sum_{x_{k}\in M_{\rho}}f(x_{k})\lambda_{x_{k}} = \int_{\bf M} f dx
$$ 
for all $f \in R_{\omega}(\mathcal L)$, and hence for all $f \in {\mathbf E}_{\omega}(\mathcal L)$, as desired.

\end{proof}

\section{On the product of eigenfunctions of the Casimir operator $\mathcal{L}$ on compact
 homogeneous manifolds}

In this section, we will use the assumption that ${\bf M}$ is a compact homogeneous manifold,
and that ${\mathcal L}$ is the operator of (\ref{Laplacian}), in an essential way.

The following Theorem \ref{prodthm} plays a crucial role in our construction of Parseval frames in section 8.
Note that  some parts of the proof of this
Theorem can be found in the papers \cite{Pes00}, \cite{Pes08}.

\begin{thm}
\label{prodthm}
If ${\bf M}=G/K$ is a compact homogeneous manifold and $\mathcal{L}$
is defined as in (\ref{Laplacian}), then for any $f$ and $g$ belonging
to ${\mathbf E}_{\omega}(\mathcal{L})$,  their product $fg$ belongs to
${\mathbf E}_{4d\omega}(\mathcal{L})$, where $d$ is the dimension of the
group $G$.

\end{thm}

\begin{proof} First, we are going to show that a function $f\in
L_{2}({\bf M})$ belongs to the space ${\mathbf E}_{\omega}(\mathcal{L})$ if and
only if there exists a constant $C(f,\omega)$ such that  the
following Bernstein inequality is satisfied  for all natural $k$
\begin{equation}
\|\mathcal{L}^{k}f\|\leq C(f,\omega)\omega^{k}\|f\|.
\end{equation}
The fact that the above Bernstein inequality holds true for any
$f\in {\mathbf E}_{\omega}(\mathcal{L})$ with $C(f,\omega)=1$ is obvious.
Conversely, assume that
$$
\lambda_{m}\leq \omega<\lambda_{m+1}.
$$
If a vector $f$ belongs to the space
$\mathbf{E}_{\omega}(\mathcal{L})$ and the Fourier series
\begin{equation}
f=\sum_{j=0}^{\infty}c_{j}u_{j}\label{Fseries},
\end{equation}
$$
c_{j}(f)=<f,u_{j}>=\int_{{\bf M}}f(x)\overline{u_{j}(x)}dx,
$$
contains terms with $j\geq m+1$, then
$$
\lambda_{m+1}^{2k}\sum_{j=m+1}^{\infty}|c_{j}|^{2}\leq
\sum_{j=m+1}^{\infty}|\lambda_{j}^{k}c_{j}|^{2}\leq
\|\mathcal{L}^{k}f\|^{2}\leq
C^{2}\omega^{2k}\|f\|^{2},\>\>\>C=C(f,\omega),
$$
which implies
$$
\sum_{j=m+1}^{\infty}|c_{j}|^{2}\leq C^{2}
\left(\frac{\omega}{\lambda_{m+1}} \right)^{2k}\|f\|^{2}.
$$
In the last inequality the fraction $\omega/\lambda_{m+1}$ is
strictly less than $1$  and $k$ can be any natural number. This
shows that the series (\ref{Fseries}) does not contain terms with
$j\geq m+1$, i.e.\ the function $f$ belongs to $
\textbf{E}_{\omega}(\mathcal{L})$.

Now, since every smooth vector field on ${\bf M}$ is a differentiation
of the algebra $C^{\infty}({\bf M})$, one has that for every operator
$D_{j}, 1\leq j\leq d,$ the following equality holds for any two
smooth functions $f$ and $g$ on ${\bf M}$:
\begin{equation}
D_{j}(fg)=fD_{j}g+gD_{j}f, \>\>\> 1\leq j\leq d.
\end{equation}
Using formula (\ref{Laplacian}) one can easily verify that for any
natural $k\in \mathbb{N}$, the term
$\mathcal{L}^{k}\left(fg\right)$ is a sum of $\>\>d^{k},
\>\>\>(d=dim G),$ terms of the following form:
\begin{equation}
D_{j_{1}}^{2}...D_{j_{k}}^{2}(fg),\ 1\leq j_{1},...,j_{k}\leq d.
\end{equation}
For every $D_{j}$ one has
$$
D_{j}^{2}(fg)=f(D_{j}^{2}g)+2(D_{j}f)(D_{j}g)+g(D_{j}^{2}f).
$$
Thus, the function $\mathcal{L}^{k}\left(fg\right)$ is a sum of
$(4d)^{k}$ terms of the form
$$
(D_{i_{1}}...D_{i_{m}}f)(D_{j_{1}}...D_{j_{2k-m}}g).
$$
This implies that
\begin{equation}
\left|\mathcal{L}^{k}\left(fg\right)\right|\leq
(4d)^{k}\sup_{0\leq m\leq 2k}\sup_{x,y\in
{\bf M}}\left|D_{i_{1}}...D_{i_{m}}f(x)\right|\left|D_{j_{1}}...D_{j_{2k-m}}g(y)\right|.\label{estim3}
\end{equation}
Let us show that the following inequalities hold:
\begin{equation}
\|D_{i_{1}}...D_{i_{m}}f\|_{2}\leq
\omega^{m/2}\|f\|_{2}\label{estim1}
\end{equation}
and
\begin{equation}
\|D_{j_{1}}...D_{j_{2k-m}}g\|_{2}\leq
\omega^{(2k-m)/2}\|g\|_{2}\label{estim2}
\end{equation}
for all $f,g \in {\mathbf E}_{\omega}(\mathcal{L})$.  First, we note  that the operator
$$
-\mathcal{L}=D_{1}^{2}+...+D_{d}^{2}
$$
commutes with every $D_{j}$ (see the explanation before the
formula (\ref{Laplacian}) ).
 The same is
true for $\mathcal{L}^{1/2}$. But then
$$
\|\mathcal{L}^{1/2}f\|_{2}^{2}=<\mathcal{L}^{1/2}f,\mathcal{L}^{1/2}f>=<\mathcal{L}
f,f>=
$$
$$
-\sum_{j=1}^{d}<D_{j}^{2}f,f>=\sum_{j=1}^{d}<D_{j}f,D_{j}f>=
\sum_{j=1}^{d}\|D_{j}f\|_{2}^{2},
$$
and also
$$
\|\mathcal{L}f\|_{2}^{2}=\|\mathcal{L}^{1/2}\mathcal{L}^{1/2}f\|_{2}^{2}=
\sum_{j=1}^{d}\|D_{j}\mathcal{L}^{1/2}f\|_{2}^{2}=
$$
$$
\sum_{j=1}^{d}\|\mathcal{L}^{1/2}D_{j}f\|_{2}^{2}=\sum_{j,k=1}^{d}\|D_{j}D_{k}f\|_{2}^{2}.
$$
From here by induction on $s\in \mathbb{N}$ one can  obtain the
following equality:
\begin{equation}
\|\mathcal{L}^{s/2}f\|_{2}^{2}=\sum_{1\leq i_{1},...,i_{s}\leq
d}\|D_{i_{1}}...D_{i_{s}}f\|_{2}^{2},\ s\in \mathbb{N}, \label{eq0}
\end{equation}
which implies the estimates (\ref{estim1}) and (\ref{estim2}). For
example, to get (\ref{estim1}) we take a function $f$ from
${\mathbf E}_{\omega}(\mathcal{L})$, an $m\in \mathbb{N}$ and  do the
following
\begin{equation}
\|D_{i_{1}}...D_{i_{m}}f\|_{2}\leq \left(\sum_{1\leq
i_{1},...,i_{m}\leq
d}\|D_{i_{1}}...D_{i_{m}}f\|_{2}^{2}\right)^{1/2}=
$$
$$
\|\mathcal{L}^{m/2}f\|_{2}\leq
\omega^{m/2}\|f\|_{2}.\label{estim4}
\end{equation}
 In a similar way we obtain (\ref{estim2}).

In terminology of the paper \cite{Pes08} it means that if $f$ and
$g$ belong to ${\mathbf E}_{\omega}(\mathcal{L})$ they also belong to
$\mathbf{B}_{\sqrt{\omega}}^{2}(\mathbb{D})$, where
$\mathbb{D}=\{D_{1},...D_{d}\}$. But it was shown in \cite{Pes08},
Theorem 3.3, that $\mathbf{B}_{\sqrt{\omega}}^{2}(\mathbb{D})=
\mathbf{B}_{\sqrt{\omega}}^{\infty}(\mathbb{D})$ which means that
the following inequalities hold

\begin{equation}
\left|D_{i_{1}}...D_{i_{m}}f\right|\leq \omega^{m/2}\|f\|_{\infty}
\end{equation}
and similarly
\begin{equation}
\left|D_{j_{1}}...D_{j_{2k-m}}g\right|\leq
\omega^{(2k-m)/2}\|g\|_{\infty}.
\end{equation}
Thus, for $f,g \in {\mathbf E}_{\omega}(\mathcal{L})$ we obtain the estimate
$$
\left|D_{i_{1}}...D_{i_{m}}f\right|\left|D_{j_{1}}...D_{j_{2k-m}}g\right|\leq
\omega^{k}\|f\|_{\infty}\|g\|_{\infty}.
$$
Now, by using (\ref{estim3}) we arrive at the following estimate:
$$
\left|\mathcal{L}^{k}\left(fg\right)\right|\leq
\left(\|f\|_{\infty}\|g\|_{\infty}\right)(4d\omega)^{k}.
$$
We square both sides of this inequality and integrate over the
compact manifold ${\bf M}$. We find that, for the  constant
$C({\bf M},f,g)=\sqrt{Vol({\bf M})}\|f\|_{\infty}\|g\|_{\infty}$, the following
inequality holds for all $k\in \mathbb{N}$
$$
\|\mathcal{L}^{k}(fg)\|\leq C({\bf M},f,g)\left(4d\omega\right)^{k}.
$$
According to previous steps of the proof, this implies that the
product $fg$ belongs to ${\mathbf E}_{4d\omega}(\mathcal{L})$. The Theorem
is proved.
\end{proof}
\begin{rem} The last part of the Theorem can be proved without
referring to the paper \cite{Pes08}. Indeed, the formula
(\ref{estim3}) along with the formula (\ref{estim4}) imply the
estimate
\begin{equation}
\|\mathcal{L}^{k}(fg)\|_{2}\leq (4d)^{k}\sup_{0\leq m\leq
2k}\|D_{i_{1}}...D_{i_{m}}f\|_{2}\|D_{j_{1}}...D_{j_{2k-m}}g\|_{\infty}\leq
$$
$$(4d)^{k}\omega^{m/2}\|f\|_{2}\sup_{0\leq m\leq
2k}\|D_{j_{1}}...D_{j_{2k-m}}g\|_{\infty}.
\end{equation}
Using the Sobolev embedding Theorem and elliptic regularity of
$\mathcal{L}$, we obtain for every $s>\frac{dim {\bf M}}{2}$
\begin{equation}
\|D_{j_{1}}...D_{j_{2k-m}}g\|_{\infty}\leq
C({\bf M})\|D_{j_{1}}...D_{j_{2k-m}}g\|_{H^{s}({\bf M})}\leq
$$
$$
C({\bf M})\left\{\|D_{j_{1}}...D_{j_{2k-m}}g\|_{2}+
\|\mathcal{L}^{s/2}D_{j_{1}}...D_{j_{2k-m}}g\|_{2}\right\},
\end{equation}
where $H^{s}({\bf M})$ is the Sobolev space of $s$-regular functions on
${\bf M}$. Since the operator $\mathcal{L}$ commutes with each of the
operators $D_{j}$, the estimate (\ref{estim4}) gives the following
inequality:
\begin{equation}
\|D_{j_{1}}...D_{j_{2k-m}}g\|_{\infty}\leq
C({\bf M})\left\{\omega^{k-m/2}\|g\|_{2}+\omega^{k-m/2+s}\|g\|_{2}\right\}\leq
$$
$$
C({\bf M})\omega^{k-m/2}\left\{\|g||_{2}+\omega^{s/2}\|g\|_{2}\right\}=
C({\bf M},g,\omega,s)\omega^{k-m/2},\>\>\>s>\frac{dim\  {\bf M}}{2}.
\end{equation}
Finally we have the following estimate:
\begin{equation}
\|\mathcal{L}^{k}(fg)\|_{2}\leq
C({\bf M},f,g,\omega,s)(4d\omega)^{k},\>\>\>s>\frac{dim \ {\bf M}}{2},\>\>k\in
\mathbb{N},
\end{equation}
which leads to the same result that was obtained above.
\end{rem}

\section{ Results on General Manifolds}

In this section, we explain some general results on compact manifolds.  We start afresh in our
notation.  

Let $({\bf M},g)$ be a smooth, connected, compact Riemannian manifold without boundary
with (\cite{H3}) Riemannian measure $\mu$.
Let $L$ be a smooth, positive, second order elliptic differential operator on ${\bf M}$,
whose principal symbol $\sigma_2(L)(x,\xi)$ is positive on  $\{(x,\xi) \in T^*{\bf M}:\ \xi \neq 0\}$.
For $x,y \in {\bf M}$, let $d(x,y)$ denote the geodesic distance from $x$ to $y$.

In \cite{gmcw} and \cite{gmbes}, the first author and Azita Mayeli proved a number of general results
about the kernels of $f(t^{2}L)$ (for $f \in {\mathcal S}(\mathbb{R}^{+})$) and about frames constructed from 
such kernels, in the Besov space framework -- in the special case in which $L$ was $\Delta$, the Laplace-Beltrami
operator on ${\bf M}$.  In this section, we review some of these results, and argue that they generalize
to the situation in which $L$ is general.  (In \cite{gmcw} -- \cite{gmbes}, it was assumed that the manifold
was orientable, but this hypothesis was not actually used and may be dropped.)

First, we have:

\begin{thm}
\label{nrdglc}
(Near-diagonal localization) Say $f \in \mathcal{ S}(\mathbb{R}^{+})$ (the space of restrictions to the nonnegative real axis of Schwartz functions on $\bf{R}$).
For $t > 0$, let $K_t(x,y)$ be the kernel of $f(t^{2}L)$.  Then: \\
(a) Say $f(0) = 0$.  Then for every pair of
$C^{\infty}$ differential operators $X$ $($in $x)$ and $Y$  $($in $y)$ on ${\bf M}$,
and for every integer $N \geq 0$, there exists $C_{N,X,Y}$ as follows.  Suppose
$\deg X = j$ and $\deg Y = k$.  Then
\begin{equation}
\label{diagest}
t^{n+j+k} \left|\left(\frac{d(x,y)}{t}\right)^N XYK_t(x,y)\right| \leq C_{N,X,Y} 
\end{equation}
for all $t > 0$ and all $x,y \in {\bf M}$.\\
(b) For general $f$, the estimate (\ref{diagest}) at least holds for $0 < t \leq 1$.
\end{thm}
This was proved in section 4 of \cite{gmcw}, in the special case
in which $L = \Delta$, the Laplace-Beltrami operator on ${\bf M}$.  (A similar result to part (a) had been
proved earlier in \cite{narc1} and \cite{narc2} in the special case where ${\bf M}$ was a sphere and $f$ had compact
support away from the origin.)  The arguments in \cite{gmcw} used
certain properties of $\Delta$, which we shall now argue are shared by general $L$.  Once this is
observed, the proofs in \cite{gmcw} go through just the same as in \cite{gmcw}, and will not be
repeated here.

Let us then list the properties of $L$ which were used in section 4 of \cite{gmcw} in the special
case $L = \Delta$, and verify that they hold for general $L$.  
\begin{enumerate}
\item
For $\lambda > 0$, let $N(\lambda)$ denote the number of eigenvalues of
$L$ which are less than or equal to $\lambda$ (counted with respect to 
multiplicity).  Then for some $c > 0$,
$N(\lambda) = c\lambda^{n/2} + O(\lambda^{(n-1)/2})$.
\item
$\sqrt{L}$ is a positive elliptic pseudodifferential operator on ${\bf M}$ of order $1$.
\item
{\em If $p(\xi) \in S^{m}_{1}(\bf{R})$ (an ordinary symbol of order $m$ on $\bf{R}$, depending only
on the ``dual variable'' $\xi$), then $p(\sqrt{L}) \in OPS^m_{1,0}({\bf M})$.}
\item
Say $h \in \mathcal{S}(\bf{R})$ is even, and satisfies $\ supp\  \hat{h} \subseteq(-1,1)$, and let
$K_t^h(x,y)$ be the kernel of $h(t\sqrt{L})$.  Then for some $C > 0$, if $d(x,y) > C|t|$,
then $K_t^h(x,y) = 0$.
\end{enumerate}

\vspace{.3cm}
\#1 is a sharp form of Weyl's theorem, which is true for any second order elliptic differential 
operator on ${\bf M}$ whose principal symbol is positive on  $\{(x,\xi) \in T^*{\bf M}:\ \xi \neq 0\}$.
(\cite{Sogge93}, Corollary 4.2.2).  (Actually, weaker forms of Weyl's theorem would suffice for the
arguments in \cite{gmcw}.)\\

\#2 was used implicitly in \cite{gmcw} (specifically, in the use of \#3). 
It follows from Theorem 2 of Seeley \cite{seel}, as Seeley himself pointed out in that article.
That theorem tells us, in particular, that
if $S$ is a classical positive invertible elliptic pseudodifferential operator of order $k > 0$ on ${\bf M}$,
whose principal symbol is positive on  $\{(x,\xi) \in T^*{\bf M}:\ \xi \neq 0\}$, then $\sqrt{S}$ is
a classical positive elliptic pseudodifferential operator on ${\bf M}$ of order $k/2$.  To apply this theorem to 
obtain \#2,
one lets $P$ be the projection onto the null space of $L$, which is a finite-dimensional space of smooth functions.
Thus $P$ has a smooth kernel.  Then one notes that $\sqrt{L} = \sqrt{L+P} - P$.\\

\#3 is an immediate consequence of the main theorem of Strichartz \cite{Stric72}.  In fact, that theorem tells us, that
if $S$ is a self-adjoint elliptic operator in $OPS^1_{1,0}({\bf M})$, then $p(S) \in OPS^m_{1,0}({\bf M})$.\\

\#4 is a consequence of the finite speed of propagation property of the wave equation.  With no claim of originality,
we now explain this in some detail.  In this discussion, all differential operators and functions will be
taken to be smooth, without further comment.

Suppose that $L_1$ is a second-order differential operator on an open set $V$ in $\bf{R}^{n}$, that $L_1$ is elliptic, and
in fact that, for some $c > 0$,  its
principal symbol $\sigma_2(L_1)(x,\xi) \geq c^2|\xi|^2$, for all $(x,\xi) \in V \times \bf{R}^{n}$.  
Suppose that $U \subseteq \bf{R}^{n}$ is open, and that $\overline{U} \subseteq V$.  
Then if $\ supp F,G \subseteq K \subseteq U$, where $K$ is compact, then any solution $u$ of
\begin{align}
(\frac{\partial^2}{\partial t^2} + L_1)u = 0 \\
u(0,x) = F(x)\\
u_t(0,x) = G(x)
\end{align}
on $U$ satisfies $\ supp \ u(t,\cdot) \subseteq \{x: $ dist $(x,K) \leq |t|/c\}$.\\
\ \\
(This is a special case of Theorem 4.5 (iii) of \cite{Tay81}.  In that reference, $V = \bf{R}^{n}$.  But we can
always extend $L_1$ from $U$ to an operator on all of $\bf{R}^{n}$ satisfying the hypotheses, by letting
$L_1' = \psi L_1 + c^2(1-\psi)\Delta$ for a cutoff function $\psi \in C_c^{\infty}(V)$ which equals $1$
in a neighborhood of $\overline{U}$.)

It is an easy consequence of this that a similar result holds on manifolds. 
With $L$ as before, let us
look at the problem
\begin{align}
(\frac{\partial^2}{\partial t^2} + L)u = 0 \\
u(0,x) = F(x)\\
u_t(0,x) = G(x)
\end{align}
on ${\bf M}$.  The first thing to note is that the problem has a unique solution in any open $t$-interval
about zero.  Namely, if
$F = \sum_k a_k \varphi_k$ and $G = \sum_k b_k \varphi_k$, where the $\varphi_k$ are an orthonormal basis of
eigenfunctions of $L$, with corresponding eigenvalues $\lambda_k$, then
the solution is
\[ u(x,t) = \sum[a_k \cos (\sqrt{\lambda_k} t) + b_k \frac{\sin (\sqrt{\lambda_k} t)}{\sqrt{\lambda_k}}]\varphi_k(x), \]
where we interpret $\frac{\sin(\sqrt{\lambda_k} t)}{\sqrt{\lambda_k}}$ as $t$ if $\lambda_k = 0$.  Note also that
\begin{equation}
\label{cosis}
u = \cos(t\sqrt{L})F \mbox{ is the solution if } G \equiv 0.
\end{equation}
We then {\bf claim} that there is a $C > 0$, depending only on ${\bf M}$ and
$L$, such that if $\ supp\  F,G \subseteq K \subseteq {\bf M}$, then 
the solution $u$ satisfies $\ supp \ u(t,\cdot) \subseteq \{x: d(x,K) \leq C|t|\}$, where
now $d$ is geodesic distance.  This is proved as follows:\\
\begin{itemize}
\item
It is enough to show that, for some $\delta > 0$, the result is true whenever $|t| < \delta$.  For,
suppose that this is known.  It suffices then to show that if, for some $T > 0$, the result is true 
whenever $|t| < T$, then it is also true whenever $|t| < T+\delta$.  For this, say $T \leq t < T + \delta$,
and select $t_0 < T$ with $t-t_0 < \delta$.  
By assumption, $\ supp \ u(t_0,\cdot) \subseteq K' := \{x: d(x,K) \leq Ct_0\}$, and thus also
$\ supp\  u_t(t_0,\cdot) \subseteq K'$.  We clearly have that $u(t,x) = v(t-t_0,x)$, where
$v$ is the solution of 
\begin{align}
(\frac{\partial^2}{\partial t^2} + L)v = 0 \\
v(0,x) = u(t_0,x)\\
v_t(0,x) = u_t(t_0,x)
\end{align}
Thus 
\[ \ supp \ u(t,\cdot) = \ supp \ v(t-t_0,\cdot) \subseteq \{x: d(x,K') \leq C(t-t_0)\} 
\subseteq \{x: d(x,K) \leq Ct\}\]
 as claimed.  Similarly if $-T \geq t \geq -T-\delta$.
\item
It suffices to show that, for some $\delta, \epsilon > 0$, the result is true whenever $|t| < \delta$, and
the supports of $F$ and $G$ are both contained in an open ball $B$ of radius $\epsilon$.  For, we could then cover ${\bf M}$
by a finite number of such open balls, and choose a partition of unity $\{\zeta_j\}$ subordinate to this covering.
If we let $(f_j,g_j) = (\zeta_j f, \zeta_j g)$, and if we let $u_j$ be the solution with data $f_j, g_j$ in place
of $f,g$, then surely $u = \sum_j u_j$.  Then surely $\ supp\  u(t,\cdot) \subseteq \{x: d(x,K) \leq C|t|\}$ 
as desired.
\item
To find appropriate $\delta, \epsilon$, one need only cover ${\bf M}$ with balls $\{B_k\}$ of some radius $\epsilon$,
for which the balls $B_k'$ with the same centers and radius $2\epsilon$ are charts, on which, if we use
local coordinates, the geodesic distance is comparable to the Euclidean distance.  The existence of a 
suitable $\delta, C$ now follows at once from the aforementioned result for the wave equation on open 
subsets of $\mathbb{C}^{n}$.  This proves the ``claim''.
\end{itemize}

To prove \#4, it suffices to write (for some $c$)
\begin{equation}
\label{wavetrck}
h(t{\sqrt L})F = c\int_{-1}^{1} \hat{h}(s) \cos(s t{\sqrt L})F ds
\end{equation}
for any $F \in C^{\infty}({\bf M})$.  (This is easily verified by using the eigenfunction expansion
of $F$ and the Fourier inversion formula.)  
\#4 follows at once from (\ref{cosis}) and the ``claim''.
 
Thus we have Theorem \ref{nrdglc} for general $L$.\\

We turn now to Besov spaces.
For the rest of this section, we fix $a > 1$.  We also fix $\alpha, p, q$ with
$-\infty < \alpha < \infty$ and $0 < p,q \leq \infty$. 
We let $B_p^{\alpha q}$ be the Besov space of section 3.

We fix a finite set ${\mathcal P}$ of real $C^{\infty}$ vector fields on ${\bf M}$, 
whose elements span the tangent space at each point. 
We also fix a spanning set of the differential operators on ${\bf M}$ of degree less
than or equal to $J$ (for any fixed $J$):

\begin{equation}
\label{pmdfopdf}
{\mathcal P}^J = \{X_1\ldots X_M: X_1,\ldots,X_M \in {\mathcal P}, 1 \leq M \leq J\} \cup
\{\mbox{the identity map}\}.
\end{equation}

The following results were obtained
in Lemmas 2.4, 3.2 and 3.3 of \cite{gmbes}, again in the special case $L = \Delta$.  In the present
article, as we shall see, the technical restrictions on $l$ and $M$ in Lemmas \ref{besov1}
and \ref{besov2} below will end up playing no role,
so the reader is advised not to pay undue attention to them.

\begin{lem}
\label{fjan2}
Say $l, M$ are integers with $l \geq 0$ and $M > n$.  
Then there exists $C > 0$ as follows.

Say $\sigma, \nu \in \bf{R}$ with $\sigma \geq \nu$.

Say $x_0 \in {\bf M}$, and suppose that 
 $\varphi_1 = L^l\Phi$, where $\Phi \in C^{2l}({\bf M})$ satisfies:
\[ |\Phi(y)| \leq (1+a^{\sigma}d(y,x_0))^{-M}. \]
Also suppose $x_1 \in {\bf M}$, that $\varphi_2 \in C^{2l}({\bf M})$, and that for all $y \in {\bf M}$,
\[ |L^l \varphi_2(y)| \leq (1+a^{\nu}d(y,x_1))^{n-M}.\] 
Then, 
\[ \left |\int_{\bf M}(\varphi_1 \varphi_2)(y)d\mu(y)\right | \leq C a^{-\sigma n}(1+a^{\nu}d(x_0,x_1))^{n-M}. \]
\end{lem}

\begin{lem}
\label{besov1}
Fix $b > 0$.  Also fix
an integer $l \geq 1$ with
\begin{equation}
\label{lgqgint}
2l > \max(n(1/p-1)_+ - \alpha, \alpha).
\end{equation}
where here $x_+ = \max(x,0)$.  Fix $M$ with $(M-2l-n)p > n+1$ if $0 < p < 1$, $M-2l-n > n+1$
otherwise.

Then there exists $C > 0$ as follows.

Say $j \in \bf{Z}$. 
Write ${\bf M}$ as a finite disjoint union of 
measurable subsets $\{\tilde{E}^j_k: 1 \leq k \leq \tilde{\mathcal N}_j\}$.  Suppose:
\begin{equation}
\label{ejkntb}
\mbox{the diameter of each } \tilde{E}^j_k \mbox{ is less than or equal to } ba^{-j}. 
\end{equation}
For each $k$ with $1 \leq k \leq \tilde{\mathcal N}_j$, select any $\tilde{x}^j_k \in \tilde{E}^j_k$. 

Suppose that, for each $j \geq 0$, and each $k$, 
\begin{equation}
\label{var2lph}
\tilde{\varphi}^j_k = (a^{-2j}L)^l\tilde{\Phi}^j_k,
\end{equation}
 where $\tilde{\Phi}^j_k \in C^{\infty}({\bf M})$ satisfies the following conditions:
\begin{equation}
\label{xphip}
|X\tilde{\Phi}^j_k(y)| \leq a^{j (\deg X + n)}(1+a^{j}d(y,\tilde{x}^j_k))^{-M} \mbox{  whenever } X \in {\mathcal P}^{4l}. 
\end{equation}
Then, for every $F$ in the inhomogeneous Besov space $B_p^{\alpha q}({\bf M})$, if we let 
\[ \tilde{s}_{j,k} = \langle F, \tilde{\varphi}^j_k \rangle, \]
then
\begin{equation}
\label{besov1way}
\left(\sum_{j = 0}^{\infty} a^{j\alpha q} \left[\sum_k \mu(\tilde{E}^j_k)|\tilde{s}_{j,k}|^p\right]^{q/p}\right)^{1/q} \leq C\|F\|_{B_p^{\alpha q}}.
\end{equation}
\end{lem}

\begin{lem}
\label{besov2}
Fix $b > 0$.  Also
fix an integer $l \geq 1$ with
\begin{equation}
\label{lgqgintac}
2l > n(1/p-1)_+ - \alpha.
\end{equation}
where here $x_+ = \max(x,0)$.  Fix $M$ with $(M-n)p > n+1$ if $0 < p < 1$, $M-n > n+1$
otherwise.

If $0 < p < 1$, we also fix a number $\rho > 0$.
Then there exists $C > 0$ as follows. 

Say $j \in \bf{Z}$.  Select sets $\tilde{E}^j_k$ and points $\tilde{x}^j_k$ as in Lemma \ref{besov1}.
If $0 < p < 1$, we assume that, for all $j,k$,
\begin{equation}
\label{ejkro}
\mu(\tilde{E}^j_k) \geq \rho a^{-jn}
\end{equation}
Suppose that, for each $j \geq 0$, and each $k$, 
$\tilde{\varphi}^j_k = (a^{-2j}L)^l\tilde{\Phi}^j_k$, where $\tilde{\Phi}^j_k \in C^{\infty}({\bf M})$ satisfies the following conditions:
\[ |X\tilde{\Phi}^j_k(y)| \leq a^{j (\deg X + n)}\left(1+a^{j}d(y,\tilde{x}^j_k)\right)^{-M} \mbox{  whenever } X \in {\mathcal P}^{4l}. \]
Suppose that $\{\tilde{s}_{j,k}: j \geq 0, 1 \leq k \leq \tilde{\mathcal N}_j\}$ satisfies 
\[ (\sum_{j = 0}^{\infty} a^{j\alpha q} [\sum_k \mu(\tilde{E}^j_k)|\tilde{s}_{j,k}|^p]^{q/p})^{1/q} < \infty. \]
Then $\sum_{j=0}^{\infty} \sum_k \mu(\tilde{E}^j_k) \tilde{s}_{j,k} \tilde{\varphi}^j_k$ converges in $B_p^{\alpha q}({\bf M})$, and
\begin{equation}
\label{besov2way}
\left\|\sum_{j=0}^{\infty} \sum_k \mu(\tilde{E}^j_k) \tilde{s}_{j,k} \tilde{\varphi}^j_k\right\|_{B_p^{\alpha q}} \leq
C\left(\sum_{j = 0}^{\infty} a^{j\alpha q}\left[\sum_k \mu(\tilde{E}^j_k)|\tilde{s}_{j,k}|^p\right]^{q/p}\right)^{1/q}.
\end{equation}
\end{lem}
Again, these three results were proved in \cite{gmbes}
in the special case in which $L = \Delta$.
The arguments in \cite{gmbes} used certain properties of $\Delta$, which we shall now argue are shared 
by general $L$.  Once this is observed, the proofs in \cite{gmbes} go through just the same as in  
\cite{gmbes}, and will not be repeated here.

Let us then list the properties of $L$ which were used in the proofs of these lemmas in \cite{gmbes} in the special
case $L = \Delta$, and verify that they hold for general $L$.  

\begin{itemize}
\item
In the proof of Lemma \ref{fjan2} in the special case $L = \Delta$ (which was Lemma 2.4 in \cite{gmbes}), the
only property used of $L$ was that it was a smooth, second-order partial differential operator, which satisfied
$\langle LF, G \rangle = \langle F, LG \rangle$ for all $F,G \in C^2({\bf M})$.
\item
In the proof of Lemma \ref{besov1} in the special case $L = \Delta$ (which was Lemma 3.2 in \cite{gmbes}), the
only properties of $L$ that were used was that it was a smooth, second-order partial differential operator,
and that Lemma \ref{fjan2} above holds.
\item
In the proof of Lemma \ref{besov2} in the special case $L = \Delta$ (which was Lemma 3.3 in \cite{gmbes}), again
these properties of $L$ were used: it is a smooth, second-order partial differential operator, and
Lemma \ref{fjan2} above holds.  In addition, the following result of Seeger-Sogge \cite{SS} was used:\\

Choose $\beta_0 \in C_c^{\infty}((1/4,16))$, with the property that
for any $s > 0$, $\sum_{\nu=-\infty}^{\infty} \beta^2_0(2^{-2\nu}s) = 1$.  For $\nu \geq 1$, define
$\beta_{\nu} \in C_c^{\infty}((2^{2\nu-2},2^{2\nu+4}))$, by $\beta_{\nu}(s) = \beta_0(2^{-2\nu}s)$.
Also, for $s > 0$, define 
the smooth function $\beta_{-1}(s)$ by
$\beta_{-1}(s) = \sum_{\nu=-\infty}^{-1} \beta(2^{-2\nu}s)$.  (Note that $\beta_{-1}(s) = 0$
for $s > 4$.)
Then (\cite{SS}), for $F \in C^{\infty}({\bf M})$, $\|F\|_{B_p^{\alpha q}}$ is equivalent to 
the $l^q$ norm (actually a quasi-norm if $0 < q < 1$)
of the sequence $\{ 2^{\nu\alpha}\|\beta_{\nu}(L)F\|_p: -1 \leq \nu \leq \infty\}$.\\

(Note that the notation of \cite{gmbes} is slightly different from that of \cite{SS};
what \cite{gmbes} calls $\beta_{k-1}(s^2)$, is called $\beta_k(s)$ in \cite{SS}.)  By Theorem 4.1 of
\cite{SS}, the result does hold for general $L$, and in fact would hold if we only knew that
$L = P^2$ for some first-order elliptic, positive, classical pseudodifferential operator on ${\bf M}$.
Of course, we do know that our $L$ satisfies this condition (see our comments on Seeley's work above,
in our discussion of point \#2, following Theorem \ref{nrdglc}).
\end{itemize}

Thus we do indeed have Lemmas \ref{besov1} and \ref{besov2} for general $L$.  
In the next section, we will put this to 
use in the case where ${\bf M}$ is a compact homogeneous manifold.\\

To conclude this section, we shall continue to work on our general ${\bf M}$, and
show how Theorem \ref{nrdglc} and the Seeger-Sogge characterization of Besov spaces
can be used to obtain a description of Besov spaces in terms of best
approximations by band-limited functions.  This result gives a generalization of a part of Theorem 1.1 of  
\cite{Pes09}, where such a description was given in the case $p=2$ for manifolds of bounded geometry. 
Our arguments are analogous to those of \cite{narc2}, Proposition 5.3, where the case in which 
${\bf M}$ is the sphere was dealt with.

We need to make a few observations first.  In the situation of Theorem \ref{nrdglc} (a), it is 
easy to see from eigenfunction expansions that $f(t^2 L)$ maps distributions on ${\bf M}$ to 
distributions on ${\bf M}$.  We have:
\begin{equation}
\label{ftllp}
\mbox{If } 1 \leq p \leq \infty, \mbox{ then } f(t^2 L): L_p({\bf M}) \to L_p({\bf M}),
\mbox{ with norm bounded independent of } t.
\end{equation}
Indeed, say that $K_t$ is the kernel of $f(t^2 L)$; it suffices to observe that for some
$C > 0$, $\int |K_t(x,y)| dx \leq C$ for all $y$, and $\int |K_t(x,y)| dy \leq C$ for 
all $x$.  This however is evident from (\ref{diagest}) with $j=k=0$, since
by (21) of \cite{gmcw}, for any $N > n$ there is a $C_N$ such that
$\int [1+ d(x,y)/t]^{-N} dy \leq C_N t^n$ for all $x$.  \\

Suppose next that $\alpha  > 0$ and $1 < p \leq \infty$, $0 < q < \infty$.  Then, on ${\bf M}$,

\begin{equation}
\label{algeinp}
B^{\alpha q}_p \subseteq L_p, \mbox{ and for some } C > 0,\  \|F\|_{L_p} \leq C\|F\|_{B^{\alpha q}_p}
\mbox{ for all } F \in B^{\alpha q}_p.
\end{equation}
Indeed, recalling our original definition of Besov spaces in section 2, we see that it is enough
to prove this on $\RR^n$.  Choose the $\Phi, \varphi_{\nu}$ in (\ref{besrndf}) in such a manner that
$\Phi + \sum_{\nu = 0}^{\infty} \varphi_{\nu} = 1$ pointwise, so that this is true in ${\mathcal S}'$ as well.
From this, if $F \in {\mathcal S}'$ is such that $\sum_{\nu = 0}^{\infty} \|\check{\varphi}_{\nu}*F\|_{L_p} < \infty$,
then $F \in L^p$, and $F = \Phi * F + \sum_{\nu = 0}^{\infty} \varphi_{\nu} * F$ in $L^p$.  
But the absolute convergence of $\sum_{\nu = 0}^{\infty} \|\check{\varphi}_{\nu}*F\|_{L_p}$ follows easily
if $F \in B_p^{\alpha q}$, by H\"older's inequality if $q \geq 1$, or directly if $0 < q < 1$.  This
shows that $B^{\alpha q}_p \subseteq L_p$, and the inequality $\|F\|_{L_p} \leq C\|F\|_{B^{\alpha q}_p}$
follows similarly.\\

The argument that $\sum_{\nu = 0}^{\infty} \|\check{\varphi}_{\nu}*F\|_{L_p}$ converges absolutely
may be adapted to ${\bf M}$.
Let the $\beta_{nu}$ be as in the Seeger-Sogge result described above.  A similar argument, using their
characterization of Besov spaces, shows that, assuming $\alpha  > 0$ and $1 < p \leq \infty$, $0 < q < \infty$,
one has:

\begin{equation}
\label{cnvp}
\mbox{ If } F \in B^{\alpha q}_p({\bf M}), \mbox{ then } \sum \|\beta_{\nu}(L)F\|_p < \infty.
\end{equation}

We let ${\bf E}_{\omega}(L)$ denote the span of all eigenfunctions of $L$ with eigenvalue
less than or equal to $\omega$.  Let ${\mathcal D}'$ denote the space of distributions
on ${\bf M}$.  We note:

\begin{equation}
\label{cnvbet}
\mbox{ If } j \geq 0 \mbox{ and } F \in {\mathcal D}'({\bf M}), \mbox{ then }
\sum_{\nu = j+1}^{\infty} \beta^2_{\nu}(L)F \mbox{ converges in } {\mathcal D}'({\bf M}), 
\mbox{ and } F - \sum_{\nu = j+1}^{\infty} \beta^2_{\nu}(L)F  \in {\bf E}_{2^{2j+4}}(L).
\end{equation}

Indeed, the convergence of the series in ${\mathcal D}'$ follows from an examination
of the eigenfunction expansion of a smooth function.  Next,
let $G = F - \sum_{\nu = j+1}^{\infty} \beta^2_{\nu}(L)F$.    By the properties of the 
$\beta_{\nu}$, note that one has that $\sum_{\nu = j+1}^{\infty} \beta^2_{\nu}(s) = 1$ for 
$s \geq 2^{2j+4}$.  
If $u$ is an eigenfunction of $L$ with eigenvalue $\lambda$, we then see that $G(u) = 0$
if $\lambda \geq 2^{2j+4}$.  Let $\{u_i\}$ be an orthonormal basis for ${\bf E}_{2^{2j+4}}(L)$,
consisting of real-valued eigenfunctions, and say $G(u_i) = a_i$.  Then $G - \sum_i a_i u_i$ 
annihilates all eigenfunctions of $L$, so it must be zero, as needed.\\

For $1 \leq p \leq \infty$, if $F \in L_p$, we let
\begin{equation}
\mathcal{E}(F,\omega,p)=\inf_{G\in
\textbf{E}_{\omega}(L)}\|F-G\|_{p}.
\end{equation}

We then have:

\begin{theorem}
\label{ernrmeq}
Say $\alpha  > 0$, $1 \leq p \leq \infty$, and $0 < q < \infty$.  Then
$F \in B^{\alpha q}_{p}$ if and only if $F \in L_p$ and
\begin{equation}
\label{errgdpq}
\|F\|^A_{B^{\alpha q}_{p}} := 
\|F\|_{L_p} + \left(\sum_{j=0}^{\infty} (2^{\alpha j}{\mathcal E}(F, 2^{2j},p))^q \right)^{1/q} < \infty.
\end{equation}
Moreover,
\begin{equation}
\label{errnrmeqway}
\|F\|^A_{B^{\alpha q}_{p}} \sim \|F\|_{B^{\alpha q}_{p}}.
\end{equation}
\end{theorem}
{\bf Proof.}  Let the $\beta_{nu}$ be as above.  

We first show, for $F \in B^{\alpha q}_{p}$, that $\|F\|^A_{B^{\alpha q}_{p}} \leq  C\|F\|_{B^{\alpha q}_{p}}$.  Because of
(\ref{algeinp}), it is enough to show that 
\begin{equation}
\label{errgdpqmod}
\left(\sum_{j=2}^{\infty} (2^{\alpha j}{\mathcal E}(F, 2^{2j},p))^q \right)^{1/q} \leq C\|F\|_{B^{\alpha q}_{p}}.
\end{equation}

But by (\ref{cnvbet}), (\ref{ftllp}), and (\ref{cnvp}), for $j \geq 2$ we have

\begin{equation}
\label{emest}
{\mathcal E}(F,2^{2j},p) \leq \|\sum_{\nu = j-1}^{\infty} \beta^2_{\nu}(L)F\|_p \leq
C\sum_{\nu = j-1}^{\infty} \| \beta_{\nu}(L)F \|_p < \infty.
\end{equation}

If one recalls the Seeger-Sogge characterization of $B^{\alpha q}_{p}$ 
and the assumption that $\alpha > 0$, and if one uses a standard 
argument, one does find
$\|F\|^A_{B^{\alpha q}_{p}} \leq  C\|F\|_{B^{\alpha q}_{p}}$.  (One introduces an
operator on $\ell^q(\mathbb{N})$ with an appropriate kernel, and invokes Proposition 3.1 of
\cite{gmbes} to show that this operator is bounded on $\ell^q$.)

For the converse, say $\nu \geq 0$.
We simply note that if $G \in E_{2^{2\nu-2}}(L)$, then $\beta_{\nu}(L)G = 0$.
Thus, by (\ref{ftllp}), if $F \in L^p$, then $ \|\beta_{\nu}(L)(F)\|_p =
\|\beta_{\nu}(L)(F-G)\|_p \leq C\|F-G\|_p$.  Accordingly, 
\[ \|\beta_{\nu}(L)(F)\|_p \leq C {\mathcal E}(F,2^{2\nu-2},p) \mbox{ for } \nu \geq 1;\ \  \mbox{    and }
\|\beta_{\nu}(L)(F)\|_p \leq C\|F\|_p \mbox{ for all } \nu. \]
From this, we find at once that $\|F\|_{B^{\alpha q}_{p}} \leq C\|F\|^A_{B^{\alpha q}_{p}}$.
$\Box$

\section{Parseval frames and Besov spaces}

We now revert to the notation of sections 1 through 6.  We modify the construction of ``needlets'' in 
\cite{narc1}, to produce a Parseval frame on ${\bf M}$.

Say $a > 1$.  Choose a function $f \in C_c^{\infty}$, supported in the interval $[a^{-2},a^4]$ such that

\begin{equation}
\label{addto1}
\sum_{j=-\infty}^{\infty} |f(a^{-2j}s)|^2 = 1
\end{equation}
for all $s > 0$.  

(For example, we could choose a smooth function 
   $\Phi$ on $\mathbb{R}^{+}$ with $0 \leq  \Phi \leq 1$, with $\Phi \equiv 1$
in $[0,a^{-2}]$ and with $\Phi = 0$ in $[a^2,\infty)$, and
let $f(t) = [\Phi(t/a^2) - \Phi(t)]^{1/2}$ for $t > 0$.)

Recalling (\ref{Laplacian}), we note that the eigenspace for ${\mathcal L}$ corresponding to the 
eigenvalue $\lambda_0 = 0$ is the space of constant functions, since the $D_j$ span the 
tangent space at each point.  Let $P$ be the projection in $L_2({\bf M})$ onto
the space of constant functions.  We now apply the spectral theorem.  By \cite{gmcw}, Lemma 2.1(b), we have
\begin{equation}
\label{addto1op}
\sum_{j=-\infty}^{\infty} |f|^2({a^{-2j}\mathcal L}) = I-P,
\end{equation}
where the sum converges strongly on $L_2({\bf M})$.  (This is, in fact, easily seen, if one diagonalizes ${\mathcal L}$.)

Say now $F \in L_2({\bf M})$.  We apply (\ref{addto1op}) to $F$ and take the inner product with $F$.  We find

\begin{equation}
\label{addto1sc}
\sum_{j=-\infty}^{\infty} \|f({a^{-2j}\mathcal L})F\|^2_2 = \|(I-P)F\|^2_2
\end{equation}

Expand $F = \sum_m A_m u_m$ in terms of our eigenfunctions of $\mathcal L$.  Then
$f({a^{-2j}\mathcal L})F = \sum_m f(a^{-2j}\lambda_m)A_m u_m \in {\bf E}_{a^{2j+4}}({\mathcal L})$,
since $f(a^{-2j}\lambda_m) = 0$ if $\lambda_m \geq a^{2j+4}$.  Also 
$\overline{f({a^{-2j}\mathcal L})F} \in {\bf E}_{a^{2j+4}}({\mathcal L})$, so by 
Theorem \ref{prodthm}, the product of these two functions, $|f({a^{-2j}\mathcal L})F|^2$
is in ${\bf E}_{4da^{2j+4}}({\mathcal L})$.  Putting 
\begin{equation}
\label{rhoj}
\rho_j = a_{0}(4da^{2j+4}+1)^{-1/2}
\end{equation}
we now find from the cubature formula that
\begin{equation}
\label{cubl2}
\|f({a^{-2j}\mathcal L})F\|^2_2 = \sum_{k=1}^{{\mathcal N}_j}b^j_{k}|[f({a^{-2j}\mathcal L})F](x^j_{k})|^2,
\end{equation}
where $x^j_k \in M_{\rho_j}$, ($k = 1,\ldots,{\mathcal N}_j = N(M_{\rho_j})$), and 
\begin{equation}
\label{wtest1}
b^j_{k}\sim \rho_j^{n},
\end{equation}
in the sense that the ratio of these quantities is bounded above and below by positive constants.

Now, for $t > 0$, let $K_t$ be the kernel of $f(t^2{\mathcal L})$, so that,
for $F \in L_2({\bf M})$, 
\begin{equation}
\label{kerfrm}
[f(t^2{\mathcal L})F](x) = \int_{\bf M} K_t(x,y) F(y) d\mu(y).
\end{equation}
For $x,y \in {\bf M}$, we have
\begin{equation}
\label{ktexp}
K_t(x,y) = \sum_m f(t^2\lambda_m) u_m(x) \overline{u}_m(y).
\end{equation}

Corresponding to each $x^j_k$ we now define the functions
\begin{equation}
\label{vphijkdf}
\varphi^j_k(y) = \overline{K_{a^{-j}}}(x^j_k,y) = \sum_m \overline{f}(a^{-2j}\lambda_m) \overline{u}_m(x^j_k) u_m(y),
\end{equation}
\begin{equation}
\label{phijkdf}
\phi^j_k = \sqrt{b^j_k} \varphi^j_k.
\end{equation}
From (\ref{addto1sc}), (\ref{cubl2}), (\ref{kerfrm}), (\ref{vphijkdf}) and (\ref{phijkdf}), we find that for all $F \in 
L_2({\bf M})$,
\begin{equation}
\label{parfrm}
\|(I-P)F\|^2_2 = \sum_{j,k} |\langle F,\phi^j_k \rangle|^2.
\end{equation}
Note that, by (\ref{vphijkdf}) and (\ref{phijkdf}), and the fact that $f(0) = 0$, each
$\phi^j_k \in (I-P)L_2({\bf M})$.  

{\em Thus the $\phi^j_k$ form a Parseval frame (i.e.\ normalized tight frame) for \\ $(I-P)L_2({\bf M})$.}  
Note also that each $\phi^j_k$ is a finite linear combination of eigenfunctions of 
${\mathcal L}$, hence is smooth.  Moreover, since $f$ vanishes on $[a^4,\infty)$,
we have $\phi^j_k \equiv 0$ once $a^{-2j}\lambda_1 \geq a^4$.  Thus, for some $\Omega$ (specifically
$\Omega = \lfloor (\log_a \lambda_1/2) - 1 \rfloor$, where
$\lfloor \cdot \rfloor = $ greatest integer function), we have
\begin{equation}
\label{jfin}
\phi^j_k \equiv 0 \:\:\:\mbox{if } j < \Omega.
\end{equation}
Note that, by (\ref{rhoj}), for $j \geq \Omega$, we have
\begin{equation}
\label{rhojest}
\rho_j \sim a^{-j},
\end{equation}
in the sense that the ratio of these quantities is bounded above and below by positive constants.
By gereral frame theory, if $F \in L_2({\bf M})$, we have
\begin{equation}
\label{recon}
(I-P)F = \sum_{j=\Omega}^{\infty}\sum_k \langle F,\phi^j_k \rangle \phi^j_k = 
\sum_{j=\Omega}^{\infty}\sum_k b^j_k \langle F,\varphi^j_k \rangle \varphi^j_k,
\end{equation}
with convergence in $L_2$.

We now explain how to characterize Besov spaces on ${\bf M}$ by using out Parseval frames.  
We let $B_{p,0}^{\alpha q}({\bf M})$ be the space of 
distributions $F$ in the Besov space $B_{p}^{\alpha q}({\bf M})$, for which $F(1) = 0$.  We claim:

\begin{theorem}
\label{beshom}
With the $\varphi^j_k$ as above, for some $C > 0$ we have:\\
(a) 
Suppose that $\{s^j_k: j \geq \Omega,\ 1 \leq k \leq {\mathcal N}_j\}$ satisfies 
\begin{equation}
\label{sjkbes}
\left(\sum_{j = \Omega}^{\infty} a^{jq(\alpha - n/p)} \left(\sum_k |s^j_k|^p\right)^{q/p}\right)^{1/q} < \infty.
\end{equation}
Then 
\begin{equation}
\label{cnvbpq}
\sum_{j=\Omega}^{\infty} \sum_k a^{-nj} s^j_k \varphi^j_k \mbox{ converges in } B_p^{\alpha q}({\bf M}), \mbox{ and }
\end{equation}
\begin{equation}
\label{besaway}
\left\|\sum_{j=\Omega}^{\infty} \sum_k a^{-nj} s^j_k \varphi^j_k\right\|_{B_p^{\alpha q}} \leq
C\left(\sum_{j = \Omega}^{\infty} a^{jq(\alpha - n/p)} \left(\sum_k |s^j_k|^p\right)^{q/p}\right)^{1/q}.
\end{equation}
(b) Suppose $F \in B_p^{\alpha q}({\bf M})$.  Then
\begin{equation}
\label{beseq}
\left(\sum_{j = \Omega}^{\infty} a^{jq(\alpha - n/p)} \left(\sum_k |\langle F, \varphi^j_k \rangle|^p\right)^{q/p}\right)^{1/q}
< \infty.
\end{equation}
Moreover, the expression in (\ref{beseq}) defines a quasi-norm on $B_{p,0}^{\alpha q}({\bf M})$ which is 
equivalent to the usual quasi-norm on this space. (If $1 \leq p,q \leq \infty$, these quasi-norms are in fact norms.)\\
(c) Let $c_0 = 1/\sqrt{\mu({\bf M})}$.  Say $F \in B_{p}^{\alpha q}({\bf M})$.  Then
\begin{equation}
\label{reconbes}
F = F(c_0)+ \sum_{j=\Omega}^{\infty}\sum_k b^j_k \langle F,\varphi^j_k \rangle \varphi^j_k,
\end{equation}
with the oonvergence of the right side being in $B_{p}^{\alpha q}({\bf M})$.
(Here $F(c_0)$ means the distribution $F$ applied to the constant
function $c_0$.) \\
(d) Let ${\bf b}_p^{\alpha q}$ denote the quasi-Banach spaces of sequences
$\{s^j_k\}$ ($j \geq \Omega,\ 1 \leq k \leq {\mathcal N}_j$) satisfying 
(\ref{sjkbes}).  Then there are well-defined bounded operators $\tau: B_p^{\alpha q}({\bf M}) \to
{\bf b}_p^{\alpha q}$ and $\sigma: {\bf b}_p^{\alpha q} \to B_{p0}^{\alpha q}({\bf M})$,
given by $\tau(F) = \{\langle F,\varphi^j_k \rangle\}$,
$\sigma(\{s^j_k\}) = \sum_{j=\Omega}^{\infty}\sum_k b^j_k s^j_k \varphi^j_k$
(with convergence in $B_{p0}^{\alpha q}({\bf M})$);
and on $B_{p0}^{\alpha q}({\bf M})$,
$\sigma \circ \tau = id$.
\end{theorem}
{\em Proof.}  For each $j \geq \Omega$, let $E^j_k = M^j_k = M_{k,\rho_j}$ be the disjoint cover
of Lemma \ref{mkrho}.  

We are going to show that we can apply Lemmas \ref{besov1} and \ref{besov2} with

\[ \tilde{\varphi}^j_k = \varphi^{j+\Omega}_k,\ \tilde{E}^j_k = E^{j+\Omega}_k,\ 
\tilde{x}^j_k = x^{j+\Omega}_k,\ L = {\mathcal L}. \]
Choose $l \in \NN$ satisfying (\ref{lgqgint}) (and hence (\ref{lgqgintac}) as well).
Define $f_l(s) = f(s)/s^l$, so that $f_l$ is another $C_c^{\infty}$
function with support in $[a^{-2},a^4]$.  We have $f(s) = s^lf_l(s)$, 
and for any $t > 0$, $f(t^2{\mathcal L}) = (t^2 {\mathcal L})^l f_l(t^2 {\mathcal L})$.  If 
$K^l_t(x,y)$ is the kernel of $f_l(t^2{\mathcal L})$, then an examination of 
(\ref{ktexp}) and the corresponding equation for $K^l$ shows that
\begin{equation}
\label{klder}
\overline{K}_t(x,y) = (t^2{\mathcal L}_y)^l \overline{K}^l_t(x,y),
\end{equation}
where ${\mathcal L}_y$ means ${\mathcal L}$ applied in the $y$ variable.  Put
\begin{equation}
\label{Phdf}
\Phi^j_k(y) = \overline{K}_{a^{-j}}^l(x^j_k,y);
\end{equation}
then
\begin{equation}
\label{Phphder}
\phi^j_k = (a^{-2j}{\mathcal L})^l\Phi^j_k.
\end{equation}
Set $\tilde{\Phi}^j_k = a^{-2\Omega l}\Phi^{j+\Omega}_k$; then we have (\ref{var2lph}).  Let us
check that the other hypotheses of Lemmas \ref{besov1} and \ref{besov2} hold as well.

Since each $M_{k,\rho_j}$ is contained in a ball of radius $\rho_j/2$, since (\ref{rhojest}) and
(\ref{mkrhoway}) hold, and since $E^j_k = M^j_k = M_{k,\rho_j}$, we see that
(\ref{ejkntb}) and (\ref{ejkro}) hold for the $\tilde{E}^j_k$ (for some $b, \rho > 0$).  
Moreover, by (\ref{Phdf}) and Theorem \ref{nrdglc},
(\ref{xphip}) holds for $\tilde{\Phi}^j_k$,
up to a multiplicative constant (independent of $j$ and $y$).  (For this, the values of $l$
and $M$ are irrelevant.)

Thus we may avail ourselves of the conclusions of Lemmas \ref{besov1} and \ref{besov2}.  Note,
by (\ref{rhojest}) and (\ref{mkrhoway}), that
\begin{equation}
\label{muejksim}
\mu(E^j_k) \sim a^{-jn}.
\end{equation}

Let
\begin{equation}
\label{cjkdf}
c^j_k = a^{jn}\mu(E^j_k);
\end{equation}
then the set $\{c^j_k\}$ is bounded above and below by positive constants.

For (a), say that (\ref{sjkbes}) holds.  We find that 
\begin{equation}
\label{sjkbescmp}
\left(\sum_{j = \Omega}^{\infty} a^{jq\alpha} \left[\sum_k \mu(E^j_k)\left|
(s^j_k/c^j_k)\right|^p \right]^{q/p}\right)^{1/q} 
\sim
 \end{equation}
 $$
 \left(\sum_{j = \Omega}^{\infty} a^{jq(\alpha - n/p)} \left[\sum_k |s^j_k|^p\right]^{q/p}\right)^{1/q}
< \infty,
$$
where now $\sim$ means that the ratio of the quantities is bounded above and below by positive
constants independent of the particular collection of $\{s^j_k\}$.  Noting that
$\mu(E^j_k)(s^j_k/c^j_k) = a^{-jn}s^j_k$, we now see that part (a) of the theorem follows at
once from Lemma \ref{besov2}.

For (b), suppose first that $F \in B_p^{\alpha q}({\bf M})$.  The sum in (\ref{beseq})
is less than or equal to 
$$
C\left(\sum_{j = 0}^{\infty} a^{j\alpha q}\left [\sum_k \mu(E^j_k)\langle F,\varphi^j_k \rangle|^p\right]^{q/p}\right)^{1/q},
$$
for some $C$, which is less than or equal to $C\|F\|_{B_p^{\alpha q}}$ for some (other) $C$, by Lemma
\ref{besov1}.  To complete the proof of (b), we must obtain the reverse inequality for $F \in  
B_{p,0}^{\alpha q}({\bf M})$.  

Before doing that, let us prove (c).  By (\ref{recon}), (\ref{reconbes}) holds for 
$F \in C^{\infty}({\bf M})$, 
with convergence in $L_2$.  Note next 
that if $F \in B_{p}^{\alpha q}({\bf M})$, the right side of (\ref{reconbes})
does converge to some element, say $T(F)$, in
$B_{p}^{\alpha q}({\bf M})$.  Indeed, to see this, by (a), we need only 
check that 
$$
\left(\sum_{j = \Omega}^{\infty} a^{jq(\alpha - n/p)} 
\left[\sum_k |a^{nj}b^j_k \langle F,\varphi^j_k \rangle|^p\right]^{q/p}\right)^{1/q} < \infty.
$$  
But,
by (\ref{wtest1}) and (\ref{rhojest}), this
quantity is less than or equal to 
$$
C\left(\sum_{j = \Omega}^{\infty} a^{jq(\alpha - n/p)}\left[\sum_k |\langle F, \varphi^j_k \rangle|^p\right]^{q/p}\right)^{1/q}
$$
for some $C$, which (by the part of (b) that we have shown), is less than or equal to $C\|F\|_{B_p^{\alpha q}}$
for some (other) $C$.  Thus, by (a), the right side of (\ref{reconbes})
does converge to some element $T(F) \in B_{p}^{\alpha q}({\bf M})$, and moreover, the map 
$T: B_{p}^{\alpha q} \to B_{p}^{\alpha q}$ is {\em bounded}.
Next note that, if $F \in C^{\infty}({\bf M})$, then $T(F) = F$.  Indeed, the right side
of (\ref{reconbes}) converges to $F$ in $L_2$, hence in the sense of distributions.  But it
converges to $T(F)$ in $B_{p}^{\alpha q}$, hence also to $T(F)$ in the sense of distributions.
Thus $T(F) = F$ as claimed.  Finally, $C^{\infty}$ is dense in $B_p^{\alpha q}$ (for instance, 
by Theorem 7.1 (a) of \cite{FJ1};
the constructions in that paper show that the building blocks can be taken to be smooth).  Since
$T$ is bounded, we must have $T(F) = F$ for all $F \in B_{p}^{\alpha q}$.  This proves (c).

Now we complete the proof of (b). For $F \in B_{p,0}^{\alpha q}({\bf M})$, we have, from
(c) and then (a), that
\[ \|F\|_{B_p^{\alpha q}} = 
\left\| \sum_{j=\Omega}^{\infty}\sum_k b^j_k \langle F,\varphi^j_k \rangle \varphi^j_k\right\|_{B_p^{\alpha q}}
\leq 
C\left(\sum_{j = \Omega}^{\infty} a^{jq(\alpha - n/p)} 
\left[\sum_k |a^{nj}b^j_k \langle F,\varphi^j_k \rangle|^p\right]^{q/p}\right)^{1/q},
\]
from which
\[ \|F\|_{B_p^{\alpha q}} \leq 
C\left(\sum_{j = \Omega}^{\infty} a^{jq(\alpha - n/p)} \left[\sum_k |\langle F, \varphi^j_k \rangle|^p\right]^{q/p}\right)^{1/q}.
\]
This proves (b).

Finally, for (d), it is clearly enough to reformulate (a) by showing that in (\ref{cnvbpq}) and (\ref{besaway}), we can replace the sum
$\sum_{j=\Omega}^{\infty} \sum_k a^{-nj} s^j_k \varphi^j_k$ by 
$\sum_{j=\Omega}^{\infty} \sum_k b^j_k s^j_k \varphi^j_k$.  (Then (d) will follow at once from this, (b) and (c)).  
But this reformulation of (a) is clear from 
(\ref{wtest1}) and  (\ref{rhojest}), which imply that $b^j_k \sim a^{-nj}$, and from
(a), applied with $b^j_k a^{nj} s^j_k$ in place of $s^j_k$.
 $\Box$\\

We close by noting the relation of our frames to the group action and to dilations
of the underlying quadratic form.  Standard wavelets on the real line have the property that wavelets
on the same scale may be obtained from each other by translation, while wavelets on different scales
may be obtained from each other by appropriate translations and dilations.  As we shall argue, something
similar happens on homogeneous manifolds, at least up to constant multiples.  This discussion is in
large part adapted from \cite{gmcw} and \cite{gmfr}.

Let $T$ be the quasi-regular representation of $G$ on $L^2({\bf M})$ (see (\ref{quasi})); this is 
a unitary representation, which commutes with the self-adjoint operator ${\mathcal L}$.  
Consequently, as operators on $L^2({\bf M})$,
$f(t^2\Delta)$ commutes with elements of $G$ for any bounded Borel function $f$ on $\RR$,
and in particular, if $f \in {\mathcal S}(\RR)$, which we now assume.

Recall (\ref{vphijkdf}).  Fix $j$.  We claim that
\begin{equation}
\label{gac}
\mbox{ if } g x^j_k = x^j_{k'}, \mbox{ then } T(g)\varphi^j_k = \varphi^j_{k'}.
\end{equation}

Indeed, if $g \in G$, $F \in L^2({\bf M})$, $x\in {\bf M}$ and $t > 0$, we have
\[
\int_{\bf M} K_t(gx, gy) F(y) dy = \int_{\bf M} K_t(gx, y) F(g^{-1}y) d(y) = [f(t^2{\mathcal L})(T(g)F)](gx)\\\notag
= T(g)([f(t^2{\mathcal L})(F)])(gx);
\] 
but this is just $[f(t^2{\mathcal L})(F)](x) = \int_{\bf M} K_t(x, y) F(y) dy$, so
\begin{equation}
\label{rotinv}
K_t(gx,gy) = K_t(x,y)
\end{equation}
for all $x,y \in {\bf M}$.  This, together with (\ref{vphijkdf}), implies (\ref{gac}) at once.   
Thus, for any fixed $j$, we one can obtain all of the $\varphi_{j,k}$ 
by applying elements of the group $G$ to any one of them. (For example, on the sphere, for any fixed
$j$, all of the $\varphi_{j,k}$ are rotates of each other.)  This is then true as well for the
frame elements $\phi^j_k$, up to constant multiples (recall (\ref{phijkdf})).

As far as different scales are concerned, there is a dilation in the background.  
Recall the discussion leading to (\ref{Laplacian}).  When we pass from
the kernel of $f({\mathcal L})$ to the kernel of $f(t^2{\mathcal L})$, we are replacing the
$D_j$ by $t D_j = tD_{X_j}$, or equivalently replacing the $X_j$ by $tX_j$, or equivalently
replacing the quadratic form $Q$ by its dilate $Q/t^2$.

\end{document}